\documentclass[11pt, reqno]{amsart}
\usepackage{latexsym}
\usepackage{amsfonts, amsmath, amscd}
\usepackage{amssymb}

\usepackage[all]{xy}
\usepackage{color}
\usepackage{enumerate}
\usepackage[usenames,dvipsnames]{xcolor}
\usepackage{verbatim}
\usepackage{amsthm}    
\usepackage{hyperref}
 \usepackage{wrapfig}
 \usepackage{graphicx, psfrag}
   \usepackage{pstool}
 \usepackage{pinlabel}   
 \usepackage{float}   

    \usepackage{xspace}   
    \usepackage{tikz-cd}
\usepackage{tikz}
 \usepackage{MnSymbol}
 \usepackage{cite}

\usepackage[normalem]{ulem}     

\def\al{\alpha}

\def\ep{\varepsilon}

\def\bd{\partial}

\def\phi{\varphi}

\def\si{\sigma}

\def\w{\omega}

\def\be{\begin{equation}}
\def\ee{\end{equation}}
\def\bear{\begin{eqnarray}\begin{aligned}}
\def\eear{\end{aligned}\end{eqnarray}}
\def\best{\begin{eqnarray*}}
\def\eest{\end{eqnarray*}}

\newtheorem{theorem}{Theorem}[section]
\newtheorem{prop}[theorem]{Proposition}

\newtheorem{lemma}[theorem]{Lemma}
\newtheorem{cor}[theorem]{Corollary}

\newtheorem{defn}[theorem]{Definition}

\newtheorem{remark}[theorem]{Remark}
\newenvironment{rem}{\begin{remark}\rm}{\end{remark}}

\newtheorem{example}[theorem]{Example}
\newenvironment{ex}{\begin{example}\rm}{\end{example}}

\newtheorem{ExtLemma}[theorem]{Extension Lemma}

\def\r#1{\right#1}
\def\l#1{\left#1}
 \def\xra{\xrightarrow} 
  \def\ra{\rightarrow} 

\def\ma#1{\mathop {#1} \limits}

\def\del{\overline \partial}

\def\Bbb{\mathbb}

\def\Z{{ \Bbb Z}}
\def\R{{ \Bbb R}}

\def\Q{{ \Bbb Q}}
\def\cx{{ \Bbb C}}

\def\cal{\mathcal}

\def\A{{\cal A}}
\def\E{\cal E}
\def\F{\cal F}

\def\J{\cal J}
\def\JV{\cal{JV}}
\def\M{{\cal M}}
\def\oM{\overline{\cal M}}
\def\P{{ \cal  P}}
\def\PP{\Bbb P}

\def\SS{{\cal S}}

\def\U{{\cal U}}

\def\ov#1{\overline{#1}}
 \def\oM{\overline{M}}

\def\ind{\mathrm{index}}
\def\Aut{\mathrm{Aut}}

\def\Hom{\mathrm{Hom}}

\def\dim{\mathrm{dim}\, }

 \newcommand\Cech{\v{C}ech\xspace }
\newcommand\cHH{\check{\mathrm{H}}}  
\newcommand\sHH{{}^s{\mathrm{H}}}  
\def\H{\check{\cal {H}}}  

\title{\bf  Thin compactifications and relative fundamental classes     \vskip.2in}
\author{Eleny-Nicoleta Ionel}
\address{Stanford University,
Stanford, CA, USA.}
\email{ionel@math.stanford.edu}

\author{Thomas H. Parker}
\address{Michigan State University, East Lansing,  MI, USA. }
\email{parker@math.msu.edu}

\thanks{The research of E.I. was partially supported by the Simons Foundation Fellowship \#340899.
}

\begin{document}

\begin{abstract}
We define a notion of  relative fundamental class that applies to  moduli spaces in gauge theory and in symplectic Gromov--Witten theory. 
For universal moduli spaces    over a parameter space, the  relative fundamental class specifies an element of the   \Cech homology of  the compactification of each fiber;  it  is defined if the compactification is ``thin'' in the sense that  the boundary of the generic fiber has homological codimension at least two.
\end{abstract}
\maketitle

\vspace{5mm}

The moduli spaces that occur  in   gauge theories and  in symplectic Gromov--Witten theory are  often orbifolds that can be  compactified by adding ``boundary strata'' of lower dimension.  Often, it is straightforward to prove that  each stratum is a manifold, but   more difficult  to prove  ``collar theorems'' that describe  how these strata fit together.   The lack of  collar theorems is an impediment to applying singular homology to the compactified moduli space, and in particular to defining its fundamental homology class.   The purpose of this paper is to show that  collar theorems  are not needed to define a (relative) fundamental class   as an element of    \Cech    homology  for families of appropriately compactified manifolds.

\smallskip

 One can distinguish  two types  of homology theories.  Type I theories,  exemplified by  singular  homology, are based on finite chains and are functorial under continuous maps.  Type II theories, 
 exemplified by  Borel-Moore singular homology, are based on locally finite (possibly infinite)  chains, and are  functorial under {\em  proper} continuous maps.    We will use two  theories of the second type:  (type II) Steenrod homology $\sHH_*$ and (type II) \Cech homology $\cHH_*$.  These  have two features that  make them especially well-suited for  applications to compactified moduli spaces:

\medskip

(1)   For any closed subset $A$ of a locally compact Hausdorff space $X$, the  relative group
$ \sHH_p(X,A)$ is identified with $\sHH_p(X\setminus A)$.  As   Massey notes \cite[p. vii]{ma2}:
\begin{quote}{\small
\dots one does not need to consider the relative homology or cohomology groups of a pair $(X,A)$; the homology or cohomology groups of the complementary space $X-A$ serve that function.  In many ways these ``single space'' theories are simpler than the usual theories involving relative homology groups of pairs. The analog of the excision property becomes a tautology, and never needs to be considered.  It makes possible an intuitive and straightforward discussion of the homology and cohomology of a manifold in the top dimension, without any assumption of differentiability, triangulability, compactness, or even paracompactness!   }
\end{quote}

(2)  \Cech homology   satisfies a  ``continuity property'' (\eqref{1.Cech.ContinuityProperty} below) that allows one to define   relative fundamental classes by a limit process.

\medskip

We briefly review Steenrod and \Cech homology in Section~1.  Then, in Section~2, we apply Property (1) to manifolds $M$ that admit compactifications $\ov{M}$ whose
  ``boundary''  $\ov{M}\setminus M$ is ``thin'' in the sense that it has homological codimension at least 2.   There may be many such compactifications.  If $M$ is oriented and $d$-dimensional, every thin compactification carries a fundamental class 
  $$
 [ \ov{M}]\in \sHH_d(\ov{M};\Z)
  $$
 in Steenrod homology.  This class pushes forward under  maps $M\to Y$ that extend continuously over  $\ov{M}$, and many properties of  fundamental classes of manifolds continue to hold.

We next enlarge the setting by considering thinly compactified {\em families}.  We consider a proper  continuous map
\bear
\label{0.MP}
\xymatrix{
\ov\M \ar[d]^{\ov\pi}\\
\P
}
\eear
from a Hausdorff space to a locally path-connected Baire metric space whose generic fiber is a thin compactification in the sense of Section~2.  More precisely, as
in Definition~\ref{Def3.1}, we call \eqref{0.MP} a ``relatively thin family'' if there is a Baire second category subset $\P^*$ of $\P$  such that (i) the fiber $\ov\M_p$ over each $p\in\P^*$ is a thin compactification of a $d$-dimensional oriented manifold,  and (ii) a similar condition holds for a dense set of paths in $\P$.  Then the fiber over each $p\in\P^*$  has a fundamental class, {\em which we now regard as an element of  \Cech  homology} (see Lemma~\ref{1.3theorieslemma}).     Because $\P^*$ is dense, 
 a limiting process using Property (2) then yields a class -- now called a relative fundamental class --  in the  \Cech  homology of {\em every} fiber of $\ov{\pi}$.   This important fact, stated as Extension Lemma~\ref{extLemma}, is used repeatedly in subsequent arguments. 
We then give a precise definition of a relative fundamental class (Definition~\ref{1.defnVFC}) and  prove:
 
\vspace{3mm}
{\em  \noindent{\bf Theorem~\ref{3.existsVFC}.} 
Every thinly compactified family $\pi:\ov{\M}\to \P$ admits  a unique relative fundamental class. }
\vspace{3mm} 

\noindent The end of Section~4 explains how a relative fundamental class yields numerical invariants associated to the family.

 Section~5 describes how relatively thin families arise from Fredholm maps.  Suppose that $\pi:\M\to\P$ is a Fredholm map between Banach manifolds with index~$d$. A ``Fredholm-stratified thin compactification'' is an   extension of $\pi$  to a  proper  map  $\ov{\pi}:\ov{\M}\to \P$ such that the boundary $\SS=\ov{\M}\setminus \M$ is  stratified by Banach manifolds ${\cal S}_\alpha$ so that, for each $\alpha$, $\ov\pi$ restricts to a Fredholm map ${\cal S}_\alpha \to \P$   of index at most $d-2$ (see Definition~\ref{Def4.2}).    The Sard-Smale theorem implies that such compactifications fit into the context of Section~4:

\vspace{3mm} 
{\em  \noindent{\bf Lemma~\ref{Lemma4.3}.}
 A Fredholm-stratified thin family  is a relatively thin family.}
 \vspace{3mm}

Section~6 describes how a relative fundamental class on one thinly compactified family extends or restricts to  relative fundamental classes on related families.

 The remaining sections give examples.  In each example, we show that the relevant moduli space admits a Fredholm-stratified thin compactification.   Lemma~\ref{Lemma4.3} and Theorem~\ref{3.existsVFC} then immediately imply the existence of a  relative fundamental class.

Sections~7 and 8 apply these ideas to Donaldson theory.  Given an oriented Riemannian four-manifold $(X,g)$,  one   constructs  moduli spaces $\M_k(g)$ of  $g$-anti-self-dual $U(2)$-connections. Donaldson's polynomial invariants are defined by evaluating certain natural cohomology classes  on $\M_k(g)$  for a generic  $g$.  We show that results already present in Donaldson's work imply the existence of relative fundamental classes for the Uhlenbeck compactification $\ov{\M}_k(g)$ for {\em any} metric.

 Sections~9 and 10 give applications to Gromov--Witten theory.   Here the central  object is the moduli space of stable maps into a closed symplectic manifold $(X, \w)$, viewed as a family
\be\label{0.GW}
\M_{A,g,n}(X) \to \JV
\ee
over the space of Ruan-Tian perturbations, as described in Section~\ref{section9}.  Again, the theme is that many results  in the literature can be viewed as giving conditions under which  there exist thin compactifications of the Gromov--Witten moduli spaces 
\eqref{0.GW} over $\JV$, or  over some subset of $\JV$.   In these  situations, the results of Sections~2--6 produce  a relative fundamental class over a subset of $\JV$.  Section~10 presents two examples: the moduli space of somewhere-injective $J$-holomorphic maps, and the moduli space of  domain-fine $(J,\nu)$-holomorphic  maps.

 \medskip

 We note that John Pardon, building on the  work of McDuff and Wehrheim \cite{mw},   has constructed a virtual fundamental class on the space of stable maps for any genus and any closed symplectic manifold \cite{pardon}.   While Pardon's approach  is different  from the one presented here, both produce classes in the dual of   \Cech cohomology, and we  expect that they  are equal  whenever  both are defined.

 \medskip
 
 We thank John Morgan and John Pardon for very helpful conversations,   Mohammed Abouzaid for encouraging us to write  these ideas out in full, and Dusa McDuff for feedback on an early version of this paper.

\vspace{5mm}
\setcounter{equation}{0}
\section{Steenrod and \Cech  homologies} 
\label{section1}
\bigskip

Expositions of Steenrod homology are surprisingly hard to find in the literature.   We will use  the  type II version  of Steenrod homology that is based on ``infinite chains'', as
presented in Chapter~4 of W.~Massey's book \cite{ma2}.   We call this  simply ``Steenrod homology'' and denote it by $\sHH_*$  (Massey's notation is $H_*$ in Chapters 4-9 and $H_*^\infty$ in Chapters~10 and 11).   For background, see also \cite{ma1}, \cite{milnor},   and the introduction to \cite{ma2}.

\smallskip

 Steenrod  homology with  abelian coefficient group $G$ assigns, for each integer $p$, an abelian group $\sHH_p(X)=\sHH_p(X; G)$ to each locally compact Hausdorff space $X$, and a homomorphism $f_*:\sHH_p(X)\to \sHH_p(Y)$ to each {\em proper} continuous map.  The axioms for this homology theory  \cite[p. 86]{ma2} include:
  \begin{itemize}
\item   For each open subset $U\subseteq X$ and each $p$,  there is a  natural  ``restriction'' map 
\bear\label{defrho}
\rho_{X,U}:\sHH_p(X)\to \sHH_p(U).
\eear
\item For each closed set $\iota:A\hookrightarrow X$, there is a natural  long exact sequence 
\bear
\label{longexactsequence}
 \cdots \longrightarrow \sHH_p(A) \overset{i_*}\longrightarrow \sHH_p(X) \overset{\rho}\longrightarrow \sHH_p(X-A) \overset{\partial}\longrightarrow \sHH_{p-1}(A) \longrightarrow \cdots
\eear
\item  If $X$ is the union of disjoint open subsets $\{X_\alpha\}$, then the inclusions $\iota_\alpha:X_\alpha\to X$ induce monomorphisms in homology, and $\sHH_p(X)$ is the cartesian product
\bear
\label{cartesianproduct}
\sHH_p(X)\ =\ \prod_\alpha\   (\iota_\alpha)_*\sHH_p(X_\alpha).   
\eear
\item  For any  inverse system $\{\cdots \to Y_3\to Y_2\to Y_1\}$ of compact metric  spaces with limit $Y$, the  maps $Y\to  Y_\al$ induce a natural exact sequence \cite[Theorem~4]{milnor}
\be
\label{cech.cont} 
0\longrightarrow  {\lim}^1 \left[\sHH_{p+1}(Y_\alpha;G)\right]  \longrightarrow  \sHH_p(Y; G) \longrightarrow \varprojlim  \sHH_p(Y_\alpha;G)\longrightarrow 0.
\ee
\end{itemize}

The corresponding cohomology theory is Alexander-Spanier cohomology with compact support.  For  compact Haudorff spaces,  this is isomorphic to both Alexander-Spanier and \Cech cohomology  $\cHH^*$ \cite[p. 334]{Sp}, and there is a universal coefficient theorem  \cite[Cor. 4.18]{ma2},   
\best
 0\longrightarrow \mbox{Ext}(\cHH^{p+1}(\oM; \Z), G) \longrightarrow  \sHH_p(\oM;G) \longrightarrow \mbox{Hom}(\cHH^p(\oM,  \Z),G) \longrightarrow 0.
\eest
\medskip

 In Sections~1-4, the term ``manifold'', or ``topological manifold'' for emphasis, of dimension $d$, means a Hausdorff space in which each point has an open neighborhood homeomorphic to $\R^d$.  Any further assumptions, such as compactness, connectedness, or second countability, will be explicitly specified as needed.    Orientations can be defined as in \cite[\S 3.6]{ma2}.  One has the following facts for any oriented $d$-dimensional topological manifold and any abelian coefficient group $G$:

\begin{itemize}
\item For all $p>d$, 
\be\label{1.i}
  \sHH_p(M)=0.
 \ee
\item For each topological $d$-ball $B$ in a connected component $M_\alpha$ of $M$, 
\bear\label{1.ball}
\sHH_d(B;G)\cong G
\eear
 and 
 \bear\label{1.ii}
\rho_{BM_\alpha}: \sHH_d(M_\alpha) \to   \sHH_d(B) \qquad  \mbox{is an isomorphism.}
\eear
\item The orientation determines a fundamental class $[M]\in \sHH_d(M; \Z)$  such that for each open ball $B\subseteq M$, regarded as a manifold with the induced orientation, 
$$
\rho_{BM}[M]=[B].
$$
 More generally,  if  $N$ is an open  subset  of $M$  with the induced orientation, then 
\be\label{1.iv}
\rho_{NM}[M]=[N].
\ee
\item If $M$ has  components $\{M_\alpha\}$, the fundamental class is given under the isomorphism \eqref{cartesianproduct} by 
\be\label{1.iii}
[M]\ =\ \prod_\alpha \ [M_\alpha].
\ee
\end{itemize}
For proofs, see  \cite{ma2},   Theorems~2.13 and  3.21a,  page 112,  and  Lemma~11.6.

\medskip
 Note that \eqref{1.ball} shows  a key difference between  type~I and type~II homology theories: in a type~II homology,  a ball $B\subseteq \R^d$ has a fundamental class.  This, as well as  the existence of the restriction map \eqref{defrho}, stem from the fact that  type~II homology is constructed using chains that are dual to {\em compactly supported} cochains.   For the same reason,  type~II homology is invariant only under {\em proper} homotopies.

\bigskip

In Section~2, we work exclusively with Steenrod homology.  In Section~3, where we consider families of spaces, we pass instead to  \Cech homology, because it satisfies the following

\vspace{2mm}

\noindent{\bf Continuity Property.} For every inverse system of compact Hausdorff spaces as in \eqref{cech.cont}, the  maps $Y\to Y_\al$ induce a natural  isomorphism
\be\label{1.Cech.ContinuityProperty}
 \cHH_*(Y; G) \overset{\cong}\longrightarrow \varprojlim  \cHH_*(Y_\alpha; G)
 \ee
\cite[pages 260-261]{ES}. 
\vspace{4mm}

In general, Steenrod homology does not satisfy the continuity property  (it satisfies  \eqref{cech.cont} instead), and \Cech homology does not  satisfy the exactness axiom.
However,  for every compact Hausdorff  space $X$,  abelian group $G$,  and commutative ring $R$,  there are natural maps 
\be\label{1.3theoriesrow}
\sHH_p(X; G) \longrightarrow  \cHH_p(X;G)  \qquad \mbox{and}\qquad   \cHH_p(X;R)\longrightarrow \cHH^p(X;R)^\vee,
\ee
where  $\cHH^p(X;R)^\vee=\Hom ( \cHH^p(X;R), R)$ is the dual to \Cech cohomology  (cf. Remark 5.0.2 in \cite{pardon}).   Furthermore, when restricted to compact metric spaces and  rational coefficients, both arrows in \eqref{1.3theoriesrow}  are isomorphisms (the first  by  Milnor's uniqueness theorem \cite{milnor}), giving a theory that is both exact and continuous (cf.  \cite[p.~233] {ES}).   

\begin{lemma}\label{1.3theorieslemma}
Let $\H_*(X)$ denote one of the three possibilities:
\be\label{1.defH}
\hspace*{12mm} \H_*(X)\ =\ \begin{cases}\itemsep=3mm  
\ \cHH_*(X; \Z) \qquad  & \mbox{\rm \Cech homology, or} \\    \ \cHH^*(X; \Z)^\vee  & \mbox{\rm Dual \Cech cohomology, or }\\ \  \cHH_*(X; \Q)  & \mbox{\rm Rational \Cech  homology.}
\end{cases} 
\ee
\noindent Then there is a natural transformation $\sHH_*(X; \Z) \to \H_*(X)$ defined on the category of  compact Hausdorff spaces, and $\H_*$ satisfies the Continuity Property (i.e. \eqref{1.Cech.ContinuityProperty} holds with $\cHH_*$ replaced by $\cHH_*$).
\end{lemma}
\begin{proof}
 For any abelian group $G$,  \Cech homology satisfies \eqref{1.Cech.ContinuityProperty} while, with the same notation,   \Cech cohomology satisfies  
\be\label{1.cohomContinuity}
\cHH^p(Y, \Z) = \varinjlim_\alpha  \cHH^p(Y_\alpha; \Z)
\ee
\cite[pages 260-261]{ES}.  Hence by Proposition~5.26 in \cite{rotman},
$$
\cHH^p(Y;  \Z)^\vee  = \mbox{Hom}(\varinjlim_\alpha  \cHH^p(Y_\alpha;  \Z),  \Z)) = \varprojlim_\alpha  \mbox{Hom}(\cHH^p(Y_\alpha;  \Z),  \Z) = \varprojlim_\alpha  \cHH^p(Y_\alpha;  \Z)^\vee.
$$
\end{proof}

\medskip

Each of the possibilities in  Lemma~\ref{1.3theorieslemma}  pairs with  \Cech cohomology; there is no longer any need for Alexander-Spanier cohomology.  \Cech cohomology, of course, is different from singular  cohomology but, for  any  $G$ and any paracompact Hausdorff space $X$, there is a natural map
\be\label{1.CechSingular}
 \cHH^p(X; G) \to  H^p_{sing}(X; G)
\ee
that is an isomorphism if $X$ is a manifold, or more generally if $X$ is  locally contractible \cite[Corollaries 6.8.8 and  6.9.5]{Sp}.

\vspace{5mm}

\setcounter{equation}{0}
\section{Thin compactifications} 
\label{section2}
\bigskip

In Steenrod homology with integer coefficients, oriented open manifolds $M$ have a fundamental class, but this class is of limited use because it does not push forward under general continuous maps.  This deficiency can be rectified by considering maps that extend continuously over a compactification $\oM= M\cup S$ of $M$, and showing that  the fundamental class $[M]\in\sHH_*(M)$ extends canonically to a class $[\ov{M}]$ in $\sHH_*(\ov{M})$.   Many such compactifications are possible; making $S$ larger allows more maps to extend continuously to $\oM$, but making $S$ too large interferes with the fundamental class.  Definition~\ref{Def1.1} identifies a class of compactifications -- ``thin compactifications''  -- that is appropriate for working with fundamental classes.  These have the form
$$
\oM= M\cup S
$$
where $S$ is a space of ``homological codimension~2''.    There are no 
   assumptions about differentiability or about how $M$ and $S$ fit together, other than the requirement that  $\oM$ is a compact Hausdorff space.

\begin{defn}
\label{Def1.1}
Let $M$ be an oriented $d$-dimensional topological manifold.  A {\em thin compactification} of $M$ is a compact   Hausdorff space $\overline{M}$ containing $M$ such that the complement 
 $S=\oM\setminus M$  (the ``singular locus'') is a closed subset of codimension~2 in the sense that  
\bear\label{2.1}
\sHH_p(S)=0 \ \quad \forall\, p> d-2. 
\eear
\end{defn}

Every compact manifold is a thin compactification  of itself (with $S$ empty),  and for each  oriented  manifold of finite dimension $d\ge 2$, the 1-point compactification  is a  thin compactification.    Further examples arise  from stratified spaces of the following type (as   was communicated to us by both J.~Morgan and  J.~Pardon).

\begin{lemma}[Stratified thin compactification]
 \label{Lemma2.2}
Suppose that  an oriented  $d$-dimensional topological manifold $M$ is a subset of a compact Hausdorff space $\overline{M}$ that, as a set,  is a disjoint union
\bear\label{Def2.2eq}
 \ov{M}\ =\ M \cup \bigcup _{k\ge 2} S_k, 
\eear
 where for each $k\ge 2$, $S_k$ is a manifold of dimension at most $d-k$, and $T_k := \bigcup_{i\ge k} S_i$ is closed.  Then $\ov{M}$ is a thin compactification of $M$.
\end{lemma}

\begin{proof}
By induction on $k$, we will show that $\sHH_p(T_k)=0$ for all $p>d-k$, which implies that the singular set $S=T_2$ satisfies \eqref{2.1}.  The induction starts with $k=d+1$ ($T_{d+1}$ is empty) and descends.  For $p>d-(k-1)$,  we have $\dim S_{k-1} \le d-(k-1)<p+1$, so $\sHH_{p+1}(S_{k-1})=0$.  The     long exact sequence 
\best
\rightarrow \sHH_{p+1}(S_{k-1}) \rightarrow \sHH_p(T_{k-1}) \rightarrow  \sHH_q(T_k)\rightarrow 
\eest 
and the induction assumption then imply that $\sHH_p(T_{k-1})=0$, as required.
\end{proof}

In practice, singular strata are usually unions of a large number of strata $S_\alpha$.  One must form the $S_k$ of 
\eqref{Def2.2eq} as unions of the $S_\alpha$ and verify that $S_k \setminus S_{k-1}$ are manifolds.  One way of doing this is described in Lemma~\ref{LemmaA1}.

\medskip
\begin{ex}
\begin{enumerate}[(a)] \setlength\itemsep{4pt}
\item The closure $\ov{V}$ of a smooth quasi-projective variety $V\subset {\Bbb P}^N$ is a thin compactification.  
\item  For a nodal complex curve $C$,  the regular part $M=C^{reg}$ can be thinly compactified in three ways:  by its   1-point compactification, by  $C$, and by 
  its normalization $\tilde{C}$, which may be  disconnected.  

\item Define an  infinite chain of 2-spheres as follows.  For each $n=1, 2, \dots$, let $p_n$ be the point $(\frac{1}{n}, 0, 0)$ in  $\R^3$.  Let $S_n$ be the sphere with center  $q_n=\frac12(p_n+p_{n+1})$ and radius $R_n=|p_n-q_n|$ with the two points $p_n$ and $p_{n+1}$ removed.  Then $M=\bigcup S_n$ is an embedded 2-manifold  in $\R^3$, and  $\oM=M\cup S$ is a thin compactification with a  singular set $S=\bigcup p_n \cup (0,0,0)$ of  dimension zero.
  
\item In contrast, $M=\{ \frac1n\, |\, n\in \Z \} \subset \R$ is  a 0-manifold, but its compactification $M\cup \{0\}$ is not  thin. 
\end{enumerate}
\end{ex}

 \medskip
 
  We now come to the key point of these definitions: in Steenrod homology,  the   fundamental class of an oriented manifold $M$ extends to any thin compactification.

 \begin{theorem}
 \label{1.VFC}
 Let $M$ be an oriented $d$-dimensional manifold with fundamental class $[M]$.  Every  thin compactification   $\overline{M}$  of $M$ has a  fundamental class 
 \best
 \label{1.defFC}
[\oM]\in \sHH_d(\oM; \Z)
\eest
uniquely characterized by the requirement that 
\bear\label{2.2}
\rho_M([\oM])= [M], 
\eear
 where $\rho_M:\sHH_d(\oM; \Z)\to \sHH_d(M; \Z)$ is the map \eqref{defrho}.  
  \end{theorem}

 \begin{proof}
The exact sequence  \eqref{longexactsequence} for  the closed subset $A=S$  of $\oM$, together with \eqref{2.1}, implies that  the  map
 \bear
 \label{2.MoMisom}
  \rho_M: \sHH_\ell(\oM) \overset{\cong}{\longrightarrow} \sHH_\ell(M)
 \eear
  is an  isomorphism for all $\ell\ge d$.  Taking $\ell=d$ shows that there is a unique class $[\oM]$  satisfying \eqref{2.2}. 
 \end{proof}
 
 \medskip
 
 In general, a manifold $M$ has  many  thin compactifications, each with a fundamental class related to $[M]$ by \eqref{2.2}.  If $\oM$ is one such thin compactification   with singular locus $S$, and $Z\subset M$ is a closed subset such that $Z\cup S$ has homological codimension~2, then $\oM$ is also a thin compactification of $M\setminus Z$,  and $[\ov{M\setminus Z}]=[\oM]$.   In this sense, one can ignore sets of codimension~2 in computations with fundamental classes.
  
\begin{ex}
\label{blowupexample}
 For two thin compactifications  $\ov M_1$ and $\ov M_2$  of the same  $d$-dimensional manifold $M$,  there are isomorphisms $\rho_i:\sHH_d(\ov{M}_i)\to \sHH_d(M)$, as in \eqref{2.MoMisom}, and the composition
 $$
 \rho_2^{-1}\circ\rho_1: \sHH_d(\ov{M}_1)\to \sHH_d(\ov{M}_2)
 $$
 takes  $[\ov M_1]$ to $[\ov M_2]$.  This is true even when there is no continuous map from $\ov{M}_1$ to $\ov{M}_2$. If there is a  map $f:\ov M_1\rightarrow \ov M_2$,  then $f_*[\ov M_1]=[\ov M_2]$ by the naturality of $\rho$.    In particular:
 \begin{enumerate}[(a)]
\item Let    $\pi:M_Z\to M$ be the blowup of a closed complex manifold $M$ along a  complex submanifold $Z$.  Then $M$ and $M_Z$ are two different thin compactifications of $M\setminus Z$, and   $\pi_*[M_Z]=[M]$.
\item More generally,  a rational map $X\dashedrightarrow Y$ between complex projective varieties induces an identification of $[X]$ with $[Y]$.
\item  If $\dim M\ge 2$, every thin compactification $\ov{M}$ has a map $p$ to the 1-point compactification $M^+$, and $p_*[\ov{M}]=[M^+]$.
\end{enumerate}
 \end{ex}
 
\medskip

The fundamental class of a manifold $M$ need not push forward under a general continuous map  $f:M\to X$.  However, if $f$ extends
to a continuous map $\ov{f}:\oM\to X$ from {\em some} thin compactification $\oM$ of $M$, then $\ov{f}$ is   proper, so  induces a map $\ov{f}_*$ in Steenrod homology:
\best
\label{1.triangle}
\xymatrix{
  \sHH_d(\oM)\ar[drr]^{\ov{f}_*} \ar[d]_{\rho}^\cong &&\\
   \sHH_d(M) \ar@{-->}[rr]  & & \sHH_d(X).
}
\eest
In this situation,  $[M]$  corresponds to $[\oM]$ by \eqref{2.2}, and the class $\ov{f}_*([\oM])\in \sHH_d(X)$ serves as a surrogate for $f_*[M]$.  Alternatively, one can take a \Cech  class  $\alpha\in \cHH^d(X)$ and evaluate ${\ov{f}}{\!\ }^*\alpha$ on the image of $[\oM]$ under \eqref{1.3theoriesrow}.

\bigskip

\subsection{\sc Covering maps}  The   isomorphism \eqref{2.MoMisom} implies several statements about how fundamental classes behave under  covering maps.

 \begin{lemma}
 \label{Lemma2.6}
Suppose that $\ov{f}:\oM\to \ov{N}$ is a continuous map between thinly compactified oriented manifolds.  If $f$  restricts to a  degree~$\ell$ oriented covering  $f:M \to N$, then
\bear
\label{2.5A}
 \ov{f}_*[\oM]=\ell\, [\overline{N}].
\eear
More generally,  if  $N$ has components $\{N_\alpha\}$  and $f$ restricts to a  degree $\ell_\alpha$ cover over some nonempty open ball $U_\alpha$ in each $N_\alpha$ 
 then, in the notation of \eqref{cartesianproduct} and \eqref{1.iii},  
\bear
\label{2.5B}
\ov{f}_*[\ov{M}]\ =\  \prod_\alpha \, \ell_\alpha\,  [\ov{N}_\alpha].
\eear
 Here $\ell_\alpha=0$ if $f^{-1}(U_\alpha)$ is empty.
\end{lemma}

  \begin{proof}
 First assume that  $M$ and $N$ are both connected.   
Fix an  open ball $U\subseteq N$ so that $f^{-1}(U)$ is the disjoint union of $\ell$ open balls $V_1, \dots V_\ell$.   In this situation, there is an isomorphism $\rho_U: \sHH_d(N)\to \sHH_d(U)$ as in   \eqref{1.ii},  and similar isomorphisms $\rho_i: \sHH_d(M)\to \sHH_d(V_i)$ for each $i$.   These fit into a
 commutative diagram  
$$
\xymatrix{
 \sHH_d(\ov{M})   \ar[d]_{\ov{f}_*}    \ar[r]^{\cong}_{\rho_M} &  \sHH_d(M) \ar[d]_{f_*}    \ar[rr]_{(\rho_1, \dots, \rho_\ell)} & & \bigoplus_i\sHH_d( V_i)    \ar[d]_{f_*}     \ar[r]^{\cong} & \Z\oplus \cdots \oplus \Z \ar[d]_\phi    \\
\sHH_d(\ov{N})    \ar[r]^{\cong}_{\rho_N} &  \sHH_d(N)    \ar[rr]^{\cong}_{\rho_U} &&  \sHH_d(U) \ar[r]^{\cong}   &  \Z
}
$$
where $\phi(a_1, \dots, a_\ell) = \sum a_i$, where $\rho_M$ and $\rho_N $ are isomorphisms by \eqref{2.MoMisom}, and where the  first two squares commute by the naturality of $\rho$.  Restricting the diagram to generators gives  \eqref{2.5A}.

In general, for each component $N_\alpha$ of $N$,  $f^{-1}(U_\alpha)$ is the disjoint union of components $V_{\alpha \beta}$, and \eqref{2.5A} applies to each restriction  $f_{\alpha \beta}=f|_{V_{\alpha\beta}}$, and the homologies of $\ov{M}$ and $\ov{N}$ are  cartesian products  as in  \eqref{cartesianproduct}.  This implies   \eqref{2.5B} with $\ell_\alpha=\sum_\beta \deg f_{\alpha\beta}$, and \eqref{2.5A} if all $\ell_\alpha$ are equal to $\ell$.
\end{proof}

\smallskip 
\begin{ex}  Lemma ~\ref{Lemma2.6} applies to branched covers of complex analytic varieties.
\end{ex}

\medskip

\subsection{\sc Components.}  Suppose that an oriented manifold $M$ has finitely many connected components $M_\al$, and that $\ov M$ is a thin compactification of   $M$ with singular locus $S$.  We then have:

\begin{lemma}  
\label{lemma2.7}
For each $\alpha$, $\ov M_\al= M_\al \cup S$ is a thin compactification of $M_\al$,  and 
\bear\label{Lemma2.6eq}
[\ov M] = \ma\sum_ \al \; [\ov M_\al].
\eear
\end{lemma} 
\begin{proof} 
The first statement holds because $\ov M_\al= M_\al \cup S$ is a closed, hence compact, subset of $\ov M$ and $S$ satisfies \eqref{2.1}.
The disjoint union $\bigsqcup \ov M_\al$ is therefore  another thin  compactification of $M$, and $[\bigsqcup \ov{M}_\al]= \sum_\al [\ov M_\al]$.
Moreover, the identity $M \rightarrow M$ extends to a continuous map $\iota: \bigsqcup \ov M_\al \rightarrow \ov M$.  Lemma~\ref{Lemma2.6} then gives $\iota_*[ \bigsqcup \ov{M}_\al]= [\ov M]$, and hence \eqref{Lemma2.6eq}. 
\end{proof} 

\bigskip
\subsection{\sc Thin Compactifications with boundary}
It is useful  to extend the notion of thin compactifications     to   manifolds $M$ with boundary $\partial M$.
 
\begin{defn}
\label{defthinwithboundary}
 A {\em thin compactification} of $(M, \partial M)$ is a compact Hausdorff pair $(\ov M, \ov{\bd M})$ containing $(M, \bd M)$ such that
\begin{enumerate}[(i)]
\item  $S =\ov M \setminus M$ is a closed subset of $\ov M$ of codimension 2,
\item $S'=\ov{\bd M} \setminus \bd M$ is a closed subset of $\ov{\bd M}$ of codimension 2, and
\item $S'\subseteq S$.  
\end{enumerate}   
\end{defn}
Note that (ii) implies  that  $\ov{\partial M}$ is a   thin compactification  of $\partial M$, while (iii) implies that  the interior  $M^0=M\setminus \partial M$  is a subset of $\oM\setminus \ov{\partial M}$  and that $\partial M = M\cap \ov{\partial M}$.
The  exact sequence \eqref{longexactsequence} of such a  pair $(\ov{\bd M}, \ov M)$ is, in part, 
\bear
\label{1.pairexactseq}
 \sHH_{d}(\ov M)  \xra{\rho}  \sHH_d(\oM \setminus \ov{\partial M}) \xra{\bd} \sHH_{d-1}(\ov {\bd M}) \xra{\iota_*} \sHH_{d-1}(\ov M).
\eear
When $M$ is oriented, there is an induced orientation on $\partial M$, and the interior  $M^0$ carries a fundamental class $[M^0]=[M\setminus \partial M] \in  H_d(M^0)$.  This is related to the fundamental class $[\bd M]$ of $\partial M$ by 
\bear
\label{1.partialmu}
\partial[M^0] = [\bd M] \in \sHH_{d-1}(\partial M),
\eear
where $\partial$ is the boundary operator in the sequence \eqref{longexactsequence} for the pair $(M, \partial M)$ (see \cite[Theorem~11.8]{ma2}, being mindful of orientations and noting the change of notation $H_p\mapsto H^\infty_p$ on page 302).

 \begin{lemma}
 \label{1.10}
A thin compactification  $(\ov M, \ov{\partial M})$ of an oriented  $d$-dimensional manifold with boundary $(M, \bd M)$ has a natural fundamental class $[\ov M]\in \sHH_d(\ov M\setminus \ov{\bd M})$ such that, for the maps in \eqref{1.pairexactseq}, 
\bear\label{pseudomanifoldhomologyEq}
a)\ \bd[\ov M]= [\ov{\bd M}]  \qquad\mbox{and}\qquad b)\ \ov\iota_*[\ov{\bd M}]=0.
\eear
Furthermore, $\rho'[\oM] = [M^0]$ under the restriction to   $M^0 \subseteq \oM\setminus \ov{\partial M}$.
\end{lemma}

\begin{proof} 
Combining \eqref{1.pairexactseq} with the similar sequence for the pair $(M, \partial M)$ gives the  diagram
\best
\label{0.1.com}
\xymatrix{
0 \ar[r] & \sHH_d(\ov M)\ar[d]^{\rho_{\ov M, M}} \ar[r]^{\rho\quad}& \sHH_d(\ov M \setminus  \ov {\bd M})\ar[r]^{\bd}\ar[d]^{\rho'}& \sHH_{d-1}(\overline{\partial M}) \ar[r]^{\bar{\iota}_*} \ar[d]^{\rho_\bd} & \sHH_{d-1}(\overline{M}) \ar[d]^{\rho} &
 \\
0 \ar[r] & \sHH_d(M)\ar[r]^{\rho\quad} & \sHH_d(M\setminus\partial M) \ar[r]^{\bd} & \sHH_{d-1}(\partial M) \ar[r]^{\iota_*} &\sHH_{d-1}(M)
}
\eest
where the rows are exact and the vertical maps are restriction maps to open subsets.   Using   properties~3b, 4b, and 4c listed on page 86 of \cite{ma2}, one sees that the 
three squares are commutative.  The first and third vertical arrows are  isomorphisms by parts (i) and (ii)  of  Definition~\ref{defthinwithboundary}, and  the exact sequence  \eqref{longexactsequence} for  the pair  $(\oM, S)$ shows that $\rho$ is an injection.  The Five Lemma then implies that $\rho'$ is an isomorphism.

We can define $[\ov M]\in \sHH_d(\ov M \setminus\ov {\bd M})$ uniquely by  the requirement that
\best
\rho'[\ov M]=[M^0].
\eest
Then  (\ref{pseudomanifoldhomologyEq}a)  follows from   \eqref{1.partialmu}  and the uniqueness of  \eqref{2.2}, while (\ref{pseudomanifoldhomologyEq}b)  follows from exactness of the top row of the diagram.
\end{proof}
 
\medskip

\begin{ex}
\begin{enumerate}[(a)] 
\item If $\ov{X}$ is a thin compactification of a manifold $X$  of dimension $d\ge 1$,  then the cone $C\ov{X}$ on $\ov{X}$ is a thin compactification of the cone on $X$ minus its vertex.
\vspace{-10mm}
\item \begin{minipage}[t]{4in}
 In the picture, $\ov{M}$ is the union of a cone on  $S^2$ and a cylinder $S^2\times [0,1]$, intersecting at one point $p$.  Then   the complement of the cone point $p$ is a manifold with boundary, and $\ov{M}$ satisfies the conditions of Definition~\ref{defthinwithboundary} with $S=S'=\{p\}$.
 \end{minipage}
 \begin{minipage}{2.3in}
\vspace*{10mm }
\begin{center}
\begin{figure}[H] 
\labellist
\small\hair 2pt
\pinlabel $p$ at 20 55  
\pinlabel $M_0$ at 10 22  
\pinlabel $M_1$ at 134 34
\pinlabel $M_1'$ at 134 81
\endlabellist 
\includegraphics[scale=.7]{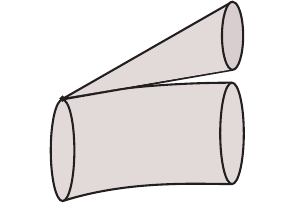}
\end{figure}
 \end{center}
\vspace*{-10mm }
\end{minipage}
 \end{enumerate}
 \end{ex}

\subsection{\sc Cobordisms.} Lemma~\ref{1.10} can be applied to cobordisms.   A   {\em thin  compactified cobordism} between $\ov{M}_0$ and $\ov{M}_1$ is a compact Hausdorff pair $(\ov W, S)$ such that 
\begin{enumerate}
\item[(i)]  $W=\ov W\setminus S$ is an oriented cobordism between two manifolds $M_0$ and $M_1$.
\item[(ii)] $\ov M_i\subset \ov W $ is a thin compactification of $M_i$ for $i=0,1$, and  $\ov M_0$ is disjoint from $\ov M_1$. 
\item[(iii)] $\sHH_{k+1}(S) = 0$ for all  $k\ge d-2$.
\end{enumerate}

 \begin{cor}
 \label{2.cobordismlemma}
Suppose that $W$ is an oriented topological cobordism between $d$-dimensional manifolds $M_0$ and $M_1$.  If $W$ 
admits a thin compactification $\overline{W}$,   then  the fundamental classes of $\oM_0$ and $\oM_1$ represent the same class in $\overline{W}$:
 \bear\label{corbordismequality}
(\iota_0)_*[\oM_0] \ =\  (\iota_1)_*[\oM_1] \qquad \mbox{ in $\sHH_d(\overline{W})$},
\eear
where $\iota_0, \iota_1$ are the inclusions of $\oM_0$ and $\oM_1$ into $\overline{W}$.
\end{cor}

 \begin{proof} The hypothesis means that $W$ is an oriented topological manifold with boundary $\bd W=M_1 \sqcup -M_0$  and that  $(\ov W, \ov{\partial W})$ is a thin compactification of $(W, \bd W)$,  where $\ov {\bd W}= \ov M_1 \cup \ov M_0$. 
 Then Lemmas~\ref{lemma2.7} and \ref{1.10} apply, and  (\ref{pseudomanifoldhomologyEq}b) becomes  \eqref{corbordismequality}.
 \end{proof}

\vspace{5mm}

\setcounter{equation}{0}

\section{Relatively thin families and the extension Lemma}
\label{section3}
\bigskip

The notion of thin compactification has a relative version. Consider a continuous map  $\pi:\M\to \P$  between Hausdorff spaces, which we regard  as a family of spaces (the fibers of $\pi$) parameterized by  $\P$.  A compactification of this family is a Hausdorff space $\ov\M$ with maps 
\bear
\label{3.1}
\xymatrix{ 
\M\ar@{^(-->}[r]\ar[d]_\pi& 
\ov\M \ar@{-->}[dl]^{\ov\pi}\\
\P
}
\eear
where the horizontal arrow is  an inclusion of $\M$ as an open subset, and $\ov\pi$ is  continuous and proper.  The fibers of $\M$ and $\ov\M$ over a point  $p\in\P$ are denoted $\M_p$ and  $\ov\M_p$ respectively; these may be empty because we are not assuming that $\ov\pi$ is surjective.  

  To extend the notion of a thin compactification to families, one might require that the fiber $\ov\M_p$ be a thin compactification of $\M_p$ for every $p\in \P$.   The aim of this section is to show that it is enough to use a weaker notion, in which the  fiber  is required to be thin only   for generic points  $p\in\P$.

 In the following definition, the term  ``second category  subset'' means a countable intersection of open dense subsets.  We will assume that $\P$ has two properties:
 \begin{enumerate} \setlength\itemsep{4pt}
\item[(a)]  $\P$ is a locally path-connected metric space, and\\[-3mm]
\item[(b)]  $\P$ is a Baire space, i.e.  every second category subset of  $\P$  is dense in $\P$.
\end{enumerate}
By the  Baire Category Theorem,   both (a) and (b) hold if $\P$ is  a metrizable separable Banach manifold.

  The space of paths in $\P$ is the set of continuous maps $\gamma: [0,1]\to \P$ with the $C^0$ topology.  For each such $\gamma$, the pullback of $\ov\M$ by  $\gamma$ is a space 
$$
\ov\M_\gamma\ =\ \big\{ (x,y)\in [0,1]\times \ov\M\, \big|\, \gamma(x)=\ov\pi(y)\big\}.
$$
 There is an associated pullback diagram
\bear
\label{3.square}
\xymatrix@=5mm{
\ov{\M}_\gamma  \ar[d]_{\ov\pi_\gamma} \ar[r]_{\hat{\gamma}} & \ov{\M} \ar[d]^{\ov\pi}   \\
[0,1]\ar[r]_\gamma & \P,
} 
\eear
with natural embeddings $\iota_0:\ov{\M}_p\to \ov{\M}_\gamma$, $\iota_1:\ov{\M}_q\to \ov{\M}_\gamma$ of the fibers over the endpoints.

\begin{defn}\label{Def3.1} 
 A {\em  relatively thin family}  of  relative dimension $d$ is  a proper continuous map 
\bear\label{D.rel.thin}
\ov\pi:\ov\M\to \P
\eear 
 from a Hausdorff space $\ov{\M}$  to a space $\P$ satisfying (a) and (b) above,  such that there exists a
  second category subset $\P^*\subseteq \P$  satisfying:
\begin{enumerate}
\item[(I)]  for each $p\in\P^*$,  the fiber $\ov\M_p$ over $p$ is a thin compactification of a $d$-dimensional oriented  topological manifold $\M_p$.\\[-2mm]
\item[(II)]  for each $p,q\in\P^*$, there is a   second category subset of paths from $p$ to $q$ such that, for each $\gamma$ in this subset,  $(\ov\M_\gamma, \; \ov\M_p \sqcup \ov \M_q)$ is a thin compactification of an oriented cobordism from $\M_p$ to $\M_q$. 
\end{enumerate}
\end{defn}

The assumptions  on $\P$ ensure that  $\P^*$ is  dense in $\P$.   Relatively thin families  often appear as compactifications:

\begin{defn}\label{Def3.1B} 
A  {\em thin compactification} of a  family $\pi:\M\to \P$ is a relatively thin family \eqref{D.rel.thin}  together with  an embedding
 as in Diagram~\eqref{3.1}.
\end{defn}

\medskip

The lemmas below use elementary topological arguments to show that  assumptions (I) and (II)  imply the existence  and uniqueness of a consistent relative fundamental class.  In subsequent sections, we will use  the Sard-Smale theorem to obtain (I) and (II).  
\smallskip

By  Lemma~\ref{1.3theorieslemma}, Assumption (I) implies that  for each $p\in\P^*$ there is an associated fundamental class 
\be
\label{3.3inCech}
[\overline\M_p] \in  \sHH_d(\overline\M_p, \Z)
\ee
in the integral Steenrod homology.   Corollary~\ref{2.cobordismlemma} and Assumption~(II) imply that this association has the consistency property  
 \bear\label{3.consistnency1}
(\iota_0)_*[\oM_p] \ =\  (\iota_1)_*[\oM_q] \qquad \mbox{ in $\sHH_d(\ov\M_\gamma)$},
\eear
along a dense of paths $\gamma$ from $p$ to $q$.

\medskip

{\em We now pass from Steenrod to \Cech homology} using the natural transformation in Lemma~\ref{1.3theorieslemma}.   The fundamental class \eqref{3.3inCech} in Steenrod homology determines a fundamental class, still denoted $[\ov\M_p]$, in \Cech homology.  Thus  there is an association
\bear\label{3.mu1}
p \mapsto [\overline\M_p] \in  \H_d(\overline\M_p)
\eear
for each   $p\in\P^*$  that satisfies the consistency property \eqref{3.consistnency1} in $\H_d(\ov\M_\gamma)$.  To proceed,  it is helpful to temporarily move to  a general context  that does not involve fundamental classes (as done in Definition~\ref{3.association} and Lemma~\ref{extLemma}).   We will return to \eqref{3.mu1} in Section~\ref{section4}.

\begin{defn}\label{3.association}
Let $\H_*$ be as in Lemma~\ref{1.3theorieslemma}.   We say an association 
$$
 p \mapsto \al_p \in  \H_*(\overline\M_p)
$$
 is {\em consistent along a path $\gamma$} from $p$ to $q$ if the images of $\al_p$ and $\al_q$ become equal in the homology of $\ov\M_{\gamma}$:
 \bear\label{3.consistent}
(\iota_0)_*\al_{p}\ =\ (\iota_1)_* \al_q \quad \mbox{in} \quad \H_*( \ov\M_{\gamma}). 
\eear 
\end{defn}

\medskip

We can now apply the continuity  property  \eqref{1.Cech.ContinuityProperty}  to extend any such consistent association  to all $p\in\P$:

\begin{ExtLemma}\label{extLemma}
\label{extensionlemma}  Let $\ov\pi:\ov\M\rightarrow \P$ be a proper continuous map from a Hausdorff space to a locally path-connected metric space $\P$. 
Suppose that there is a  dense subset $\P^*$ of $\P$ and an assignment
\bear\label{lemma3.4eq1}
p \mapsto \al_p\in \H_*(\ov{\M}_p)
\eear
defined for $p\in \P^*$   and  consistent along  paths  in a dense subset of the space of paths in $\P$ from $p$ to $q$ for   each $p,q\in\P^*$.
Then \eqref{lemma3.4eq1} extends to all $p\in\P$ so that \eqref{3.consistent} holds for all paths $\gamma$ in $\P$, and this extension is unique.
\end{ExtLemma}

 \begin{proof} Fix a point $p\in \P$, and let $B_k$ be the ball of radius $1/k$ centered at $p$. Using the definition of locally path-connected, one can inductively choose a sequence of path-connected open neighborhoods  $U_k$ of $p$ with  $U_k\subset B_k$ and $U_{k+1}\subset U_k$, for all $k\ge 1$. Then each $U_k$ contains a dense set of values $q\in\P^* \cap U_k$ for which \eqref{lemma3.4eq1}  is defined. Moreover, any two values in $\P^* \cap U_k$ can be joined by a path   in  $U_k$ which, by assumption, can be perturbed, keeping the endpoints fixed, to a path   in $U_k$ for which \eqref{3.consistent} holds.

 Choose any sequence $p_k\in U_k\cap \P^*$ (so $p_k$ converge to $p$) and paths 
 $\gamma_k:  [0,1] \to \P$ 
 from $p_k$ to $p_{k+1}$ satisfying \eqref{3.consistent} and whose image is in  $U_k$.
For  each $m\ge 1$, set $K_m=[0, \tfrac{1}{m}]$, and define a  ``segmented''  path  $\phi_m: K_m \to \P$ by $\phi_m(0)=p$ and 
$$
\phi_m(t) =\gamma_k\left(\tfrac{1}{t}-k\right)\quad \mbox{ for $t\in  \left[\frac{1}{k+1}, \ \frac{1}{k}\right]$, \ \ $k\ge m$.}
$$

Then each $\phi_m$ is a proper continuous map whose image is a path through the points $p_k=\phi_m(1/k)$ for $k\ge m$.  The pullback spaces $\ov\M_{\phi_m}$ (defined as in \eqref{3.square}) form  a nested  sequence of compact Hausdorff spaces whose intersection is the compact  space $\ov\M_p$.  There are also   natural inclusions $\iota_{km}:\ov\M_{p_k} \to \ov\M_{\phi_m}$ for each $k\ge m$.  Applying the consistency condition  \eqref{3.consistent}  inductively, one sees that the class
$$
(\iota_{km})_*\alpha_{p_k}\ \in\ \H_d(\ov\M_{\phi_m}).
$$
is independent of $k$ for $k\ge m$.  These homology classes are consistently related by the  inclusions $\ov\M_{\phi_{m_1}}\hookrightarrow \ov{\M}_{\phi_{m_2}}$ for $m_1\ge m_2$, so define an element    of the inverse  limit
\best
\label{2.VFClimit}
\varprojlim\, (\iota_{km})_*\al_{p_k} \in\varprojlim_m  \H_d(\ov{\M}_{\phi_m}).
\eest
 By the continuity property \eqref{1.Cech.ContinuityProperty}, this determines a unique   \Cech homology class
 \bear
\label{2.VFClimit2}  
\al_p \in \H_d( \ov\M_p)
\eear
which, at this point,    depends on the choices of the $p_k$ and the $\gamma_k$.

Next, fix an arbitrary continuous path $\gamma:[0,1]\to\P$ from $p\in \P$ to $p'\in\P$.  Choose  segmented
 paths  $\phi_m$ and $ \phi'_m:K_m\to\P$  as above that limit to  $p=\phi_m(0)$ and $p'=\phi'_m(0)$, respectively. Then, for each $k$, choose a path  
$\si_k$ from $p_k$ to $p_k'$ that  lies in the $1/k$ neighborhood of $\gamma$, and for which \eqref{3.consistent} holds  (specifically, $U_k\cup \gamma\cup U_k'$ is path-connected, so contains a path from $p_k\in\P^*$ to $p'_k\in\P^*$ which, by assumption, can be perturbed to the desired path $\si_k$).   For each $m$,  let $L_m\subset\R^2$ denote the ``ladder'' consisting of the union of the segments:
\medskip
\begin{center} \begin{tabular}{ll}
$I_m=\{(0,y)\ |\  0\le y\le1/m\}$ \hspace{2cm}  &   $I_0=\{(x,0)\ |\   0\le x\le1\}$ \\[2mm]
 $I'_m=\{(1,y)\ |\   0\le y\le1/m\} $  &   $J_k=\{(x, \tfrac{1}{k})\ |\  0\le x\le1\}, \ \ k\ge m.$
 \end{tabular}\end{center}
\medskip
 
\medskip

\begin{wrapfigure}[5]{r}{0.5\textwidth}
\end{wrapfigure}  
 
\noindent Now let $\Phi_m:L_m\to\P$ be the continuous map  whose restriction (i) to $I_0$ is $\gamma$ (after identifying  $I_0$ with $[0,1]$), and whose restrictions 
\vspace{2mm}
\begin{enumerate}
\item[(ii)] to $I_m$ is   $\phi_m$,  \\[-2mm]
\item[(iii)]  to $I_m'$ is   $\phi'_m$, and \\[-2mm]
\item[(iv)]  to each  $J_k$ is $\sigma_k$.  
\end{enumerate}


 \vspace{-3.2cm}
 \begin{figure}[ht!]
\labellist
\small\hair 2pt
\pinlabel $p$ at 97 65  
\pinlabel $p_1$ at   136 111
\pinlabel $p_2$ at 117 98
\pinlabel $p_3$ at 105 85
\pinlabel $p'$ at 119 52  
\pinlabel $p_1'$ at 199 85
\pinlabel $p_2'$ at 175 71
\pinlabel $p_3'$ at 158 62
\pinlabel $\Phi_m$ at 74 103
\pinlabel \mbox{Ladder $L_m$} at 23 57
\pinlabel $\gamma$ at 80  55
\endlabellist 
\hspace*{.5\textwidth} 
\includegraphics[scale=1]{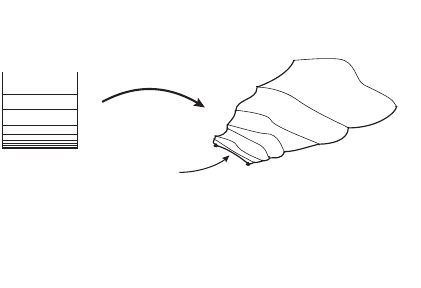}
\end{figure}

\vspace{-16mm}

Each $L_m$ is compact, so $\Phi_m$ is proper, and the pullback spaces $\ov\M_{\Phi_m}$ are   a nested sequence of compacta whose intersection is $\ov\M_\gamma$.   Again   the consistency condition  \eqref{3.consistent}  implies that,  for $k\ge m$,  the classes
\best  
(\iota_{km})_*\alpha_{p_k}, (\iota'_{km})_*\alpha_{p'_k}\ \in\ \H_d(\ov\M_{\Phi_m}).
\eest
are equal and independent of $k$, and hence form an inverse system that defines an element
\bear\label{3.ladder=}
\al_\gamma \ =\ \varprojlim\, (\iota_{km})_*\al_{p_k} \ =\ \varprojlim\, (\iota'_{km})_*\al_{p'_k} \  \in \H_d( \ov\M_\gamma). 
\eear

Recall that the class \eqref{2.VFClimit2}  depends on the choice of the points $p_k$ and the connecting paths $\gamma_k$.  But given another choice $\{p'_k, \gamma'_k\}$,  we can construct  ladder maps  $\Phi_m$ for the constant path $\gamma(t) \equiv p$.  For constant paths, $\ov\M_\gamma$ is equal to $\ov\M_p\times [0,1]$, so there 
  is a projection $\rho:\ov\M_\gamma\to \ov\M_p$.  Applying $\rho_*$ to  \eqref{3.ladder=} then  shows the class \eqref{2.VFClimit2} constructed from the two choices are equal.

With this understood,   the consistency condition \eqref{3.consistent} along $\gamma$ follows simply  by comparing \eqref{2.VFClimit2} and \eqref{3.ladder=}.

 Finally, to check uniqueness, assume $\al'$ is another extension which agrees with $\al$ on $\P^*$ and satisfies \eqref{3.consistent} for all paths $\gamma$ in $\P$. Pick any point $p\in \P$ and   segmented paths $\phi_m: K_m\to \P$ as above. Then for any $k\ge m$,  the inclusions induce equalities
\best
\al'_p \ =\ (\iota_{km})_*\al'_{p_k}\ =\ (\iota'_{km})_* \al_{p_k} \ =\  \alpha_p
\eest
in $\H_d(\ov{\M}_{\phi_m})$.  Therefore, again by continuity, we have
\best
\al'_p\ =\  \varprojlim_m \ (\iota'_{km})_*\al'_{p} \ =\  \varprojlim_m \ (\iota_{km})_*\al_{p_k} \ =\ \al_p
\eest
as elements of $\varprojlim  \H_d(\ov{\M}_{\phi_m})= \H_d( \ov\M_p)$.   Thus the extension is unique.
 \end{proof}

 \vspace{5mm}
\setcounter{equation}{0}
\section{Relative fundamental classes} 
\label{section4}
\bigskip

 We now return to the homology theory \eqref{1.defH} and define relative fundamental classes for relatively thin families. The definition is axiomatic, and we prove both existence and uniqueness. 

\begin{defn}
\label{1.defnVFC}

A {\em relative fundamental class}  for the relatively thin family \eqref{D.rel.thin}  of relative dimension~$d$ associates to each $p\in \P$ an element
\best\label{3.5[M]}
 [\ov{\M}_p]^{rel}\in \H_d(\ov{\M}_p)
\eest
such that,  for some choice of $\P^*$ as in Definition~\ref{Def3.1}, 
\begin{itemize} \setlength\itemsep{4pt}
\item[{\bf A1.}]  (Normalization) For each $p\in \P^{*}$, $[\ov{\M}_p]^{rel}$ is the fundamental class $[\ov{\M}_p]$.
\item[{\bf A2.}]  (Consistency) For every path $\gamma$ in $\P$ from $p$ to $q$, 
\bear\label{4.A2}
(\iota_0)_* [\ov{\M}_p]^{rel} = (\iota_1)_*[\ov{\M}_q]^{rel} \quad \mbox{in} \quad \H_*( \ov\M_{\gamma}).
\eear
\end{itemize}
 \end{defn}

\bigskip

Note that a relative fundamental class is not {\em a single}  class, but rather is a consistent {\em collection} of classes. It assigns a   $d$-dimensional class \eqref{3.5[M]} to  {\em every} fiber $\ov\M_p$, including those that are not thinly compactified manifolds, and those whose dimension is not $d$.  Similarly,  the consistency condition \eqref{4.A2} is a  {\em collection}  of equalities,  one for each  path in $\P$.
  The proof  of Theorem~\ref{3.existsVFC} below shows how $[\ov{\M}_p]^{rel}$ is defined at each $p$   as a limit of the fundamental classes  of the  fibers $\ov\M_p$ for $p$ in the dense set $\P^*$.
 
 Of course, the relative fundamental class depends on the  relatively thin family \eqref{D.rel.thin}, and in particular on its  relative dimension $d$. {\em A priori}, it also depends on the second category set $\P^*$, but we show next that it does not.

\bigskip

 Using the  terminology of  Definitions~\ref{Def3.1}   and \ref{1.defnVFC},  our main result  can be stated simply:

\begin{theorem}
\label{3.existsVFC}  
A relatively  thin  family  $\ov\pi:\ov\M\to \P$  admits  a unique relative fundamental class.  This class satisfies satisfies {\bf A1 }and {\bf A2} in Definition~\ref{1.defnVFC}  
for each choice of  the second category set $\P^*$ in Definition~\ref{Def3.1}, and  is independent of the choice of  $\P^*$. 
\end{theorem}

\begin{proof}  
For each $p\in\P^*$, the fiber $\ov\M_p$ is a thin compactification of an oriented $d$-manifold, and we define  $[\ov{\M}_p]^{rel}$ to be its fundamental class.
As in \eqref{3.mu1},  properties I and II of Definition~\ref{Def3.1}  imply that  the association
\bear\label{4.1.1}
p \mapsto [\ov{\M}_p]^{rel}
\eear
has the consistency property \eqref{3.consistnency1}.  Thus 
the  Extension Lemma~\ref{extLemma} applies, giving a unique extension of \eqref{4.1.1} to all $p\in\P$ that satisfies   the consistency condition Axiom~A2.

  To show independence of $\P^*\!$, suppose that  a relatively thin family  satisfies conditions I and II of Definition~\ref{Def3.1} for two second category sets ${\cal Q}^*$ and ${\cal Q}^{**}$. Then it also satisfies these conditions for the second category set $\P^{*}={{\cal Q}^*}\cap {\cal Q}^{**}$. The sets $\P^*, {\cal Q}^*$ and ${\cal Q}^{**}$ each  define a relative fundamental class, and these three classes are equal for all $p$ in  dense set $\P^{*}$. By the uniqueness in the Extension Lemma~\ref{extLemma}, they must  agree for all $p\in \P$. 
\end{proof}

\medskip

A relative fundamental class can be used to define numerical invariants.  For each $p\in\P$, there is  map
\bear\label{4.8}
I_p\,:\,  \cHH^d(\ov\M, \Z)\to \Z
\eear
defined on  a \Cech  cohomology class $\alpha\in \cHH^*(\ov{\M})$  by
\be
\label{3.Invariants}
I_p(\alpha)\ =\ 
\left \langle \alpha,\  [\ov\M_p]^{rel}\right \rangle.
\ee
Here we are implicitly restricting $\alpha$ to the fiber $\ov\M_p$, and the pairing is well defined because $\ov\M_p$ is compact.

\begin{cor}
\label{Cor4.last} 
For a relatively thin  family $\ov\pi:\ov\M\to \P$   the map \eqref{4.8}  is independent of  $p$ on each  path component  of $\P$.
\end{cor}
\begin{proof}
  Given points $p$ and $q$ in the same  path component, fix a path $\gamma:[0,1]\to\P$ from $p$ to $q$.  Pushing the consistency condition \eqref{4.A2} forward by the homology map induced by the proper map $\hat\gamma$ in diagram \eqref{3.square} shows that $[\ov\M_p]^{rel}$ is homologous to $[\ov\M_q]^{rel}$ in  $\cHH_d(\ov\M)$.   Hence $I_p(\alpha)$ is equal to $I_q(\alpha)$ for all cohomology classes $\alpha$.
\end{proof}

 \vspace{5mm}
\setcounter{equation}{0}
\section{Fredholm Families}
\label{section5}
\bigskip

 In many gauge theories, the universal moduli space admits a compactification that is stratified by  Banach manifolds in the manner described in Definition~\ref{Def4.2} below.  If so, and more generally  if such a stratification exists over an open dense subset of the parameter space, one can obtain a relative fundamental class using 
the Sard-Smale theorem and  Theorem~\ref{3.existsVFC}.

 In this  and later sections, the  term ``Banach manifold'' means a metrizable separable $C^l$  Banach manifold, finite or infinite dimensional.  Such manifolds are second countable and paracompact.
We say that a property holds ``for generic $p$'' if it holds for all $p$ in some second category subset of $\P$.  We will consider Fredholm maps 
\bear
\label{4.1}
\xymatrix{
\M \ar[d]^{\pi}\\
\P
}
\eear
between  Banach manifolds, which we again regard as a family parameterized by  $\P$ and, to emphasize this viewpoint, call it a ``Fredholm family''.  Such a map has an associated Fredholm index $d$, and we assume that 
\bear\label{4.llarge}
l > \max(d+1,0).
\eear

  The Sard-Smale theorem shows that  the generic fibers of $\pi$ are manifolds of dimension~$d$.  It also yields a similar statement about generic paths in the Banach  manifolds $\Omega(p,q)$ of $C^1$ paths $\gamma:[0,1]\to\P$ from $p=\gamma(0)$ to $q=\gamma(1)$.  The precise statements are as follows.

\begin{theorem}
\label{4.SStheorem}
For a Fredholm map \eqref{4.1} of index $d$ that satisfies \eqref{4.llarge}, 
\begin{enumerate} \setlength\itemsep{4pt}
\item[(a)] The set  $\P_0^{reg}$ of regular values of $\pi$   is a second category  subset of $\P$,   and for each $p\in\P_0^{reg}$,  the fiber $\M_p=\pi^{-1}(p)$ is a manifold of dimension $d$,  and is empty if $d<0$.
\item[(b)]  For each $p,q\in\P_0^{reg}$,   there is  a second category  subset of  $\Omega(p,q)$ consisting of   paths $\gamma$ for which  the pullback space  $\M_{\gamma}$ (cf. \eqref{3.square})  is manifold of dimension $d+1$.
\end{enumerate}
\end{theorem}

\begin{proof}
Statement (a) is the  Sard-Smale theorem; see Section~1 of  \cite{s}.  For $(b)$, set $\Omega= \Omega( p,q)$ and  let $\ep:[0,1]\times\Omega\to \P$ be the evaluation map $\ep(t,\gamma)=\gamma(t)$.   The pullback of \eqref{4.1} by $\ep$ is a  map $\ep^*\!\M\to [0,1]\times \Omega$.  Composing with the projection to $\Omega$ yields a Fredholm map  $\ep^*\!\M\to \Omega$ whose fiber over $\gamma\in \Omega$ is $\M_\gamma$.   Part (b) follows by applying part (a) to this map, as explained, for example, in Sections 4.3.1 and 4.3.2 of  \cite{DK}.
\end{proof}

 The data \eqref{4.1} also determines a real line bundle $\det(d\pi)$ over $\M$ --- the determinant line bundle of the Fredholm map $\pi$ --- whose restriction to each regular fiber $\M_p$, $p\in \P_0^{reg}$, is the orientation bundle  $\Lambda^dT\!\M_p$.  We will always assume that \eqref{4.1} has a {\em relative orientation} specified by a nowhere zero section of  $\det(d\pi)$.   We will use the term  {\em oriented Fredholm family} to mean a  Fredholm map \eqref{4.1} together with a choice of a relative orientation.
 
 \smallskip

 Given an oriented Fredholm family, we can consider compactifications as in Section~3 which are stratified by Fredholm families.  In fact, in the applications given in Sections~7-10 below, the relevant  compactifications will have the  following structure.

 \begin{defn}
\label{Def4.2} 
A {\em Fredholm-stratified thin  family} of index~$d$ is proper continuous map $\overline\pi:\overline{\M}\to\P$  
from a Hausdorff space $\ov\M$ which,  as a set, is a disjoint union
$$
 \overline\M = \M \cup \bigcup_{k=2}^\infty \SS_k
$$
such that   
\begin{enumerate}
\item[(a)] The restriction of $\ov\pi$ to $\M$ is  an index~$d$ oriented Fredholm family $\pi:\M\to\P$.
\item[(b)] For each $k\ge 2$,   the restriction of $\ov\pi$ to $\SS_k$ is an  index $d-k$ Fredholm family  $\pi_k:\SS_k\to\P$. 
\item[(c)]  ${\cal T}_k = \bigcup_{i\ge k} \SS_i$ is closed in $\ov\M$ for each $k$. 
\end{enumerate}
 
 \noindent We then  say that  $\ov\pi:\ov\M\to \P$ is a {\em  Fredholm-stratified thin compactification} of the Fredholm family $\pi$ with top stratum $\M$ and strata ${\cal S}_k$.
 \end{defn}

 \bigskip
 
   The first key observation is  that  Fredholm-stratified thin  families  fit into the context of the previous section:  the  Sard-Smale theorem implies that they are relatively thin families in the sense of  Definition~\ref{Def3.1}.

 \begin{lemma}
 \label{Lemma4.3} 
 A Fredholm-stratified thin family  is a relatively thin family  
with $\P^*$ equal to the set of regular values defined in  \eqref{5.Preg} below. 
 \end{lemma} 
\begin{proof}
By assumption,  $\P$ is a  Banach manifold, so is locally path-connected.  
Apply the Sard-Smale Theorem to  \eqref{4.1} and to each map $\pi_k: \SS_k\to \P$, and intersect the corresponding second category sets of regular values.  The result is a single second category subset 
\bear
\label{5.Preg}
\P^{reg}\subseteq\P
\eear
 whose points are simultaneous regular values of $\pi$ and all $\pi_k$; we call these  {\em regular values of $\ov\pi$}. 

 For  each  regular $p\in\P^{reg}$,  the fiber $\overline\M_p$ of $\ov\pi:\overline\M\to \ \P$ is   stratified as in \eqref{Def2.2eq}, so is a thin compactification of $\M_p$  by Lemma~\ref{Lemma2.2}.
Thus  Assumption~I of Definition~\ref{Def3.1} holds.

Similarly,  for any   $p, q\in\P^{reg}$, the Sard-Smale theorem shows that  there is  a   second category subset of the space of   paths $\gamma$ in $\P$  from $p$ to $q$ for which $\gamma$ is  transverse to $\pi_k$ for all $k$, and hence the pullback 
$(\SS_k)_\gamma$ of $\pi_k$ over $\gamma$  is a manifold (with boundary) of dimension $d-k+1$. Then $\ov\M_\gamma$ is the union of $\M_\gamma$ and the manifolds $(\SS_k)_\gamma$, so  Assumption~II of Definition~\ref{Def3.1} also holds.
\end{proof}
\bigskip

 The following simple lemma provides a useful way of verifying that a given family satisfies the conditions of Definition~\ref{Def4.2}. 
 
 \begin{lemma}
\label{LemmaA1} 
Consider an index $d$ oriented Fredholm family  \eqref{4.1} of $C^l$ manifolds with $l$ satisfying \eqref{4.llarge}.   Suppose 
that there exists a Hausdorff space $\overline{\M}$ containing $\M$ as an open set and an extension of $\pi$ to a proper continuous map $\ov\pi:\overline{\M}\to\P$ such that 
 \begin{enumerate}[(a)]
\item $\ov\M$ can be written as a disjoint union of sets $\{\SS_\al | \al\in\A\}$ indexed by a finite set $\A$ with  $0\in\A$ and $\SS_0=\M$.
\item Each $\SS_\al$ is a manifold, and $\pi_\al=\ov{\pi}|_{\SS_\al}$ is a $C^l$ Fredholm map $\SS_\alpha\to\P$ of index $d_\al$.

\item $d_\al \le d-2$ for all $\al \ne 0$, and  
\best
\ov \SS_\al \setminus \SS_\al \subseteq \bigcup_{\{\beta\, |\, d_\beta<d_\al\} } \SS_\beta.
\eest
\end{enumerate}
Then $\ov\pi: \ov\M \rightarrow \P$ is a  Fredholm-stratified thin compactification of the family $\pi:\M\to \P$.
\end{lemma}
\begin{proof} 
Condition (c) implies that the accumulation points of $\SS_\al$  that are not in $\SS_\alpha$ lie in strata of strictly smaller index. Hence for each $k$,  the union of strata of  index $d-k$ 
\best
X_{k} = \ma\bigcup_{d_\al=d- k} \SS_\al
\eest
is topologically a disjoint union of manifolds.   This means that  each $X_k$ is a manifold, and that the restriction of $\ov\pi$ to $X_k$ is a Fredholm map of index $d-k$.  It also means that $\bigcup_{i\ge k} X_i$ is closed for each $k$.   Definition~\ref{Def4.2} then applies, showing  that 
\best
\ov\M \ =\  \M \cup\ma \bigcup_{\al \ne 0}  \SS_\al \ =\   \M \cup\ma\bigcup_{k \ge 2} X_k
\eest
is a   Fredholm-stratified thin compactification of $\M\ra \P$ with strata $X_k$. 
  \end{proof}

\bigskip
 
We conclude this section with two finite-dimensional examples, both of which come from algebraic geometry.    The first  shows that the relative fundamental class can be different from the actual fundamental class even when the fiber is a manifold.

\begin{example}[\sc Elliptic Surfaces] \label{3.mfex}
{\rm  An elliptic surface is a  compact complex algebraic surface $S$ with a holomorphic projection $\pi:X\to C$ to an algebraic curve $C$ whose fiber is an elliptic curve except over a finite number of points $p_i\in C$.   The singular fibers $F_{p_i}$ are unions of rational curves, each possibly with singularities and multiplicities,  and elliptic curves with multiplicity.  The restriction of $\pi$ to the union of the smooth fibers is a Fredholm map $X^*\to C$ of index~2, and $\pi:X\to C$ is a thin compactification of $X^*$ regarded as a family over $C$.  Thus by Theorem~\ref{3.existsVFC}, every fiber $F_p$ carries a relative fundamental class
$$
[F_p]^{rel}\in \cHH_2(F_p, \Z)
$$
whose image $\iota_*[F_p]^{rel}$ in $\cHH_2(X, \Q)$ is the homology class  of the generic fiber.

In particular, if $F_p$ is a smooth elliptic fiber with multiplicity $m>1$, then $F_p$ has a fundamental class $[F_p]$, but the relative fundamental class is 
\bear
\label{3.multiplefiber}
[F_p]^{rel}=m[F_p].
\eear

}
\end{example}

\medskip

\begin{example}[\sc Lefschetz Pencils and Fibrations]
{\rm  
Consider   a complex  projective manifold $X$ with a complete linear system $|D|$ of divisors of  complex dimension at least~3.  Lefschetz showed that a generic 2-dimensional linear system $[D]$
determines  a holomorphic map $\pi:X\setminus B\to {\Bbb P}^1$,  where $B$ is the base locus of $[D]$.  The generic fiber of $\pi$ is smooth and the  other fibers have only quadratic singularities. This map $\pi$ is   therefore
Fredholm, and its index is the real dimension  $d= 2(\mbox{dim}_\cx X-1)$  of the generic fiber.  While $\pi$ does not extend continuously to $X$,   it does extend continuously over the blowup ${X}_B$ of $X$ along $B$, and $\tilde{\pi}:{X}_B\to{\Bbb P}^1$ is   a thin compactification of $X\setminus B\to {\Bbb P}^1$.  Theorem~\ref{3.existsVFC} therefore defines a relative fundamental class
$$
[F_p]^{rel}\in \cHH_d(F_p, \Z)
$$
on the  fiber $F_p=\tilde{\pi}^{-1}(p)$ over each $p\in{\Bbb P}^1$.
}
\end{example}

\vspace{5mm}
\setcounter{equation}{0}
\section{Enlarging  the parameter space}
\label{section6}
\bigskip

In gauge theories, one starts with a parameterized family of  elliptic PDEs, and considers the moduli space of solutions as a family over the space of parameters.   After completing in appropriate  Sobolev norms, this yields a map $\pi:\M\to \P$ to  a  separable Banach space $\P$ of parameters.  Often, there is a natural  compactification $\ov\M$ as in diagram~\eqref{3.1}.  

One can then hope to obtain a relative fundamental class by applying  Theorem~\ref{3.existsVFC}.  This  involves defining a stratification of
 ${\cal S}=\ov\M\setminus\M$, and proving
lemmas of two types:
\begin{enumerate}[(i)]
\item Formal dimension counts for all strata. \vspace{4pt}
\item Transversality results   showing that $\M$ and each stratum $\M_\alpha$ of ${\cal S}$ is a manifold of the expected dimension.
\end{enumerate}
In general, (ii) can be done only if the space of parameters $\P$ is sufficiently large. Thus it may be  necessary to enlarge the space of parameters in order to define relative fundamental classes.   Enlarged spaces of parameters may also be needed to  show independence of added  geometric structure, such as the choice of  a Riemannian metric used to define Donaldson polynomials (see  Section~7 and 8), and the choice of an almost complex structure used to define Gromov--Witten invariants  (Sections~9 and 10).

    When enlarging   the parameter space, some  care is needed because the   relative fundamental classes depend on the choice of $\P$ and of the thin compactification.  Thus enlarging the space of parameters may change the problem that one is trying to solve.   Lemma~\ref{lemma4.2} below gives a stability result   that ensures that a base   expansion yields a compatible relative fundamental class.

\begin{defn}
\label{baseexpansion}
 A {\em base expansion} of  a  relatively  thin compactification \eqref{D.rel.thin} 
is a relatively  thin  compactification of
 $\pi': \M' \ra \P'$ with a commutative diagram  of continuous maps
 \be\begin{gathered}
\label{6.1}
\xymatrix{
\ov\M\ar[d]^{\ov\pi}  \ \ar[r]^{F} \  &\ov{\M'}\ar[d]^{\ov\pi'}\\
\P \  \ar[r]^f\  &\P'\\
}
\end{gathered} \ee
where  there exist a second category subset   $\P^*$ of $\P$ that satisfies the conditions of Definition~\ref{Def3.1} for $\ov\pi$, and a similar subset $(\P')^*$ of $\P'$ for $\ov\pi'$, such that:
\begin{enumerate}
\item[(a)] $f(\P^*) \subseteq (\P')^*$.  \\[-3mm]
\item[(b)]  for each $p\in\P^*$\!,  $F$ restricts to   a degree~1 map  $\M_p\to \M_{f(p)}'$  between oriented  topological manifolds. 
\end{enumerate}
\end{defn}

 Note that these conditions imply that $\ov\pi$ and $\ov\pi'$ have the same relative dimension.
  
  \smallskip
  
 \begin{lemma} 
 \label{lemma4.2}
 For a  base expansion \eqref{6.1}, the relative fundamental classes of $\pi$ and $\pi'$ agree over $\P$, i.e.  for all $p\in \P$ we have
\bear\label{lemma4.2eq}
 (F_p)_*[\ov\M_p]^{rel} =[\ov{\M'}_{\!f(p)}]^{rel}  
\eear
in $\H_*(\ov{\M'}_{\!f(p)})$,  where $F_p:\ov\M_p\to \ov\M'_{f(p)}$ denotes the restriction of $F$.  
\end{lemma}

\begin{proof}
 For each $p$ in the set $\P^*$ of Definition~\ref{baseexpansion},  both $\M_p$ and $\M_{f(p)}$ are oriented topological manifolds.
  By  Definition~\ref{Def3.1}(I),  $\ov{\M_p}$ and $\ov{\M'}_{f(p)}$ are thin compactifications of $\M_p=\M'_{f(p)}$, respectively. Each carries a  fundamental class by  Theorem~\ref{1.VFC}, and these are equal to the corresponding relative fundamental class by Axiom $\mbox{A1}$ of Definition~\ref{1.defnVFC}.   Therefore,  for each $p\in \P^{*}$, 
 \best
 (F_p)_*[\ov{\M_p}]^{rel}\, =\,  (F_p)_*[\ov{\M_p}] \, =\,  [\ov{\M'}_{\!f(p)}] \, =\,   [\ov{\M'}_{\!f(p)}]^{rel},
 \eest  
where the middle equality holds by Lemma~\ref{Lemma2.6}  and Definition~\ref{baseexpansion}(b).   This then implies  \eqref{lemma4.2eq} for all $p\in \P$,  as follows.

 As  in the proof of the Extension Lemma~\ref{extensionlemma}, 
 pick nested open sets $V_k \subseteq \P$ with $\bigcap V_k = \{p\}$, and $V'_k\subseteq\P'$ with $\bigcap V'_k = \{f(p)\}$, and set $U_k=V_k\cap f^{-1}(V'_k)$.  Next,  choose a sequence $p_k\ra p$ with $p_k\in U_k\cap \P^*$,  and  segmented paths $\phi_m$ in $\P$ converging to $p$. Then, as in \eqref{2.VFClimit}, 
 \best
 [\ov\M_p]^{rel}\ =\  \varprojlim\   (\iota_{km})_*[\ov\M_{p_k}]^{rel},
 \eest
 and therefore by the naturality of \eqref{1.Cech.ContinuityProperty}  
  \best
 (F_p)_*[\ov\M_p]^{rel} \, =\,   (F_p)_*\varprojlim \   (\iota_{km})_*[\ov\M_{p_k}]^{rel}\, =\,   \varprojlim \ (F_p\circ \iota_{km})_* [\ov\M_{p_k}]^{rel}
 \eest 
 On the other hand, the images $F\circ \phi_m$ converge to $f(p)$ in $\P'$, and therefore 
 \best
  [\ov\M'_{f(p)}]^{rel} \, =\,    \varprojlim \    (j_{km})_*[\ov\M'_{f(p_k)}]^{rel},
   \eest
 where $j_{km}=F\circ\iota_{km}$ is the inclusion of $f(p_k)$ into $V'_m$.  Combining the last three displays give \eqref{lemma4.2eq} for all $p\in \P$. 
  
\end{proof}

\medskip

\begin{ex} 
\begin{enumerate}[(a)] \setlength\itemsep{4pt}
\item  If  both vertical arrows in \eqref{6.1} are Fredholm-stratified families, and  $p$ is a regular value of $\ov{\pi}$,  then the inclusion of $\ov{\M}_p\to \{p\}$ into $\ov{\pi}:\ov{\M}\to \P$ is a  base expansion.  Equation~\eqref{lemma4.2eq} becomes $[\ov{\M_p}] =[ \ov{\M_p}]^{rel}$, which is Axiom~A1 of Definition~\ref{1.defnVFC}.

\vspace{3mm}

\item Example~\ref{3.mfex}  shows the importance of condition  (a)  in Definition~\ref{baseexpansion}.  Let   $F_p$ be a smooth elliptic fiber in an elliptic surface with multiplicity $m>1$.  Then $F_p\to\{p\}$ is a thinly compactified family with $[F_p]^{rel}=[F_p]$, and the inclusion of $F_p\to\{p\}$ into $X\to C$ satisfies  all of the conditions of  Definition~\ref{baseexpansion} except  (a).  But, as in \eqref{3.multiplefiber},  the  relative fundamental class induced by  the extended family $X\to C$  is $m[F_p]$ rather than $[F_p]$.

\vspace{3mm}

\item Similarly, in Example~\ref{blowupexample}(a), the family  $\pi_Z: \pi^{-1}(Z)\to Z$ embeds into $\pi: M_Z\to M$.  In this case,  the dimensions of the generic fibers and the indices are different, so  this embedding is not a base expansion, and the two relative fundamental classes  lie in different dimensions. 
\end{enumerate}
 \end{ex}

\medskip

 Examples~(b)  and (c) above   are instances where the relative fundamental class $[\ov{\M}]^{rel}$ {\em depends on the choice of the parameter space $\P$.}  Thus it does not make sense to speak of ``the'' relative fundamental class of a single  fiber $\M_p$:  relative fundamental classes are, by their nature, associated with  relatively thin families over parameter spaces.

 \vspace{3mm}

\begin{ex}
For moduli  spaces of solutions to an elliptic differential equation, one obtains  base expansions by lowering the regularity of the parameters, for example,  by including a space   $\P^l$ of  $C^l$ parameters into the corresponding   $C^{l-1}$  space. Often, elliptic theory implies that, for sufficiently large  $l$,  all conditions in Definition~\ref{baseexpansion} are satisfied, and hence the relative fundamental class is unchanged in the sense of Lemma~\ref{lemma4.2}.   In particular,  for each smooth parameter $p\in \P^\infty = \bigcap \P^l$, the moduli space $\ov{\M}_p$ of solutions is canonically identified with the fibers  $\ov{\M}_p^l$ over $p$ in $\P^l$ for each large  $l$, and the relative fundamental classes $[\ov{\M}_p^l]^{rel}$ consistently induce a relative fundamental class on $\ov{\M}_p$.
\end{ex}
 \vspace{3mm}

In some applications, one has a family  $\ov\M\to \P$ which is not itself  Fredholm-stratified,  but whose restriction  to an open dense subset $\P^o$ of  $\P$ is Fredholm-stratified.  The next result, which will be used in Section~\ref{section8},  gives conditions under which this is sufficient to make $\ov\M\to\P$ a relatively thin family.

\begin{lemma}\label{corelemma}
Let $\ov\pi:\ov\M\to \P$ be a proper continuous map from a Hausdorff space to a  Banach manifold.  Suppose that there is an open, dense subset $\P^o$ of $\P$ such that 
\begin{enumerate}
\item[(i)] Every path in $\P$  is  a limit of paths in $\P^o$.

\item[(ii)] The restriction $\ov\pi^o:\ov\M^o\to \P^o$  of $\ov\pi$ over $\P^o$ is a Fredholm-stratified thin family of index~$d$.    
\end{enumerate}
 Then $\ov\pi:\ov\M\to \P$ is a relatively thin family of relative dimension~$d$ with $\P^*$ defined by \eqref{6.5.P*}, and therefore admits a unique relative fundamental class $[\ov{\M}_p]^{rel} \in \H_d(\ov\M_p)$ for all $p\in \P$.

\end{lemma}

\begin{proof} By Lemma~\ref{Lemma4.3}, the set
\bear\label{6.5.P*}
\P^*=\P^{reg}
\eear
of regular values of $\ov\pi^o$  is a second category subset of $\P^o$, i.e. is a countable intersection of open dense subsets.  But  open dense subsets of $\P^o$ are open and dense in $\P$  (because $\P^o$ is open and dense in $\P$), so $\P^{*}$ is also a second category subset of $\P$.

Next observe that   any path $\gamma$ in $\P$ whose endpoints $p, q$ are in $\P^*\subseteq \P^o$ is a limit of paths in $\P^o$ with the same endpoints $p, q$ as follows.  By assumption~(i), $\gamma$ is the limit of a sequence of paths $\gamma_k$  in $\P^o$ with endpoints $p_k$, $q_k$, where $p_k\ra p$ and $q_k\ra q$.  Because $p, q\in \P^o$ and $\P^o$ is open subset of a Banach manifold,   for sufficiently large $k$ we can find paths $\si_k$ in $\P^o$ from $p$ to $p_k$ converging to the constant path at $p$,  and similarly paths $\tau_k$ in $\P^o$ from $q_k$ to $q$ converging to the constant path at $q$. The concatenation of these paths is a sequence $\{\si_k \#\gamma_k \# \tau_k\}$ of paths in $\P^o$, each with  endpoints $p, q$, which limit to the path $\gamma$. 

With these observations, one sees that Definition~\ref{Def3.1} applies to $\ov\pi:\ov\M\to \P$ with this $\P^*$: 
\begin{enumerate}
\item[(i)]  Condition~I holds as in the proof of Lemma~\ref{Lemma4.3}.\\[-3mm]
\item[(ii)]  Condition~{II} holds because, again as in  the proof of Lemma~\ref{Lemma4.3},  it holds  for a dense set of paths in $\P^o$ from $p$ to $q$  described above, and this set of paths is dense in the space of paths in $\P$ from $p$  to $q$.  
\end{enumerate}
The lemma then follows by Theorem~\ref{3.existsVFC}.
 \end{proof}

\vspace{5mm}
\setcounter{equation}{0}
\section{Donaldson theory}
\label{section7}
\bigskip

Let $X$ be a  smooth, closed, oriented 4-manifold that satisfies the  Betti number condition $b^+_2(X)>1$.  Donaldson theory uses   moduli spaces  of connections to construct invariants of the smooth structure of $X$.    This section  and the next  describe how  Donaldson's polynomial invariants  fit  into the context of the previous sections.   We follow Donaldson's exposition in Sections~5.6 and 6.3 of  \cite{D2}.

\medskip

Let  $E\to X$ be a $U(2)$ vector bundle with  first Chern class $c_1=c_1(E)$ and instanton  number  $k= (c_2(E)-\frac14 c_1^2(E))[X]$.  Fix a connection $\nabla^0$ on $\Lambda^2E$.   After completing in appropriate Sobolev norms (see, for example, Section 4.2 of \cite{DK}), we obtain three separable Banach manifolds:  a space ${\cal A}={\cal A}_E(\nabla^0)$ of  connections  on $E$  that induce $\nabla^0$ on $\Lambda^2E$,    a   space    ${\cal R}$  of Riemannian metrics on $X$, and the   group ${\cal G}$ of gauge transformations of $E$  with determinant 1.  Furthermore, ${\cal G}$ acts  smoothly on ${\cal A}$,  the orbit space   ${\cal B} = {\cal A}/{\cal G}$ is  metrizable, and the subset ${\cal B}^{irred}\subset{\cal B}$ of irreducible connections is also a separable Banach manifold.

 A pair $(A, g)$  in ${\cal A}\times {\cal R}$  is called an  {\em instanton}  if its curvature $F^A$ satisfies 
 $*ad(F^A)=- ad(F^A)$,  where $*$ is the Hodge star operator on 2-forms for the metric $g$. The  universal moduli space  
 $\M_E\subset {\cal B}\times {\cal R}$ is the set of all ${\cal G}$-equivalence classes $([A], g)$ of instantons  for $A\in{\cal A}_E$.  Up to isomorphism,  $\M_E$   depends on the bundle  $E$ only though the pair $(k,c_1)$, and is independent of the connection $\nabla^0$ (see page 146 of \cite{D2}). 
  
  Now fix  $c_1$ and  consider the sequence of moduli spaces $\M_k$ associated with bundles $E$ with instanton number $k$ and this fixed $c_1$.  Projection onto the second factor is a map 
 \bear
\label{ASDmodulidiagram}
\xymatrix{
\M_k\ar[d]^{\pi}\\
{\cal R} 
}
\eear
whose restriction to  $\M_k^{irred}=\M_k\cap ({\cal B}^{irred}\times {\cal R})$ is a smooth  Fredholm map of index $2d_k$, where  $d_k$ is given in terms of the Betti numbers $b_1(X)$ and $b^+_2(X)$ of $X$ by
\be
\label{6.d_k}
d_k= 4k-\tfrac32 \big(1-b_1(X)+b^+_2(X)\big).
\ee
This Fredholm family is oriented by the choice of a homology orientation for $X$ \cite[7.1.39]{DK}. 
\medskip

Let  $\M_k(g)$ denote the fiber of $\M_k$ over a metric $g\in {\cal R}$.   We say that $c_1$ is {\it odd} if it represents a class in $H^2(X; \Z)/\mbox{Torsion}$ that is not divisible by 2.

\begin{lemma}
\label{lemma7.1}
Suppose that  $b^+_2(X)>1$ and   $c_1$ is  odd. Then there is an open dense subset ${\cal R}^o$  of ${\cal R}$ such that
 \begin{enumerate}
\item[(i)]  For each $g\in{\cal R}^o$ and each $0\le j\le k$,  the fiber  $\M_j(g)$ contains no reducible connection. \\[-3mm]
\item[(ii)]   Every path in ${\cal R}$  is the limit of paths in ${\cal R}^o$.   
\end{enumerate}
\end{lemma}

\begin{proof}
This follows directly from the discussion on page 147 of \cite{D2} and  Corollary~4.3.15 of \cite{DK}.  Note that the assumption that $c_1$ is odd  implies that  the space ${\cal A}_E$  contains no flat connections \cite[Section 5.6]{D2}. 
\end{proof}

 \begin{lemma}
 \label{donlemma}
Under the hypotheses of Lemma~\ref{lemma7.1}, the map \eqref{ASDmodulidiagram} extends to a proper continuous  map $\ov\pi:\ov{\M}_{k}\to {\cal R}$ whose restriction over ${\cal R}^o$  is a  Fredholm-stratified thin compactification of  $\pi^o: \M_k|_{\cal R^o}\to {\cal R}^o$. 
 \end{lemma}
 
 \begin{proof}
  We follow the notation and discussion in Section~4.4 of \cite{DK}.  
 Using the topology of weak convergence  (as defined by Condition~4.4.2 in  \cite{DK}),   one  defines the Uhlenbeck compactification $\ov\M_k$ by setting
\be\label{6.UC}
\ov\M_{k}\ =\ \M_{k}\cup \SS,   
\ee
where $\SS$ is the union of the strata $S_{jk}=\M_{k-j}\times {\rm Sym}^j(X)$ for $0<j< k$ (noting that $\M_0$ is empty because there are no flat connections).  Then $\ov{\M}_k$ is paracompact and metrizable \cite[Section 4.4]{DK},   and $\pi$ extends to a map $\ov\pi:\ov\M_k\to {\cal R}$ whose restriction to each stratum is Fredholm.

The proof is completed by applying Lemma~\ref{LemmaA1}.  For this,  it suffices to define a stratification  on $\ov\M_k$, different from the one in \eqref{6.UC},  whose restriction  $\ov{\M}^o_k=\ov\M_k|_{\cal R^o}$ satisfies the hypotheses of  Lemma~\ref{LemmaA1}.

The  new strata are labeled by partitions.
 A partition is a non-increasing  sequence $\alpha=(\alpha_1, \dots, \alpha_\ell)$ of positive integers;  its length  $\ell(\alpha)=\ell$ and its degree $|\alpha|=\sum \alpha_i$ satisfy $\ell(\alpha)\le |\alpha|$.  We also consider $(0)$ to be a partition with $\ell(0)=|(0)|=0$.   Let ${\cal{P}}_k$ be the set of all partitions $\alpha$ with $|\alpha|\le k$.  Define the {\em level} of $\alpha$ to be 
\be\label{A.defLambda}
\Lambda(\alpha) = 2|\alpha|-\ell(\alpha),
\ee
and note that $\Lambda(\alpha)\ge 0$ with equality if and only if $\alpha=(0)$.

Given a four-manifold $X$ and  an integer $k\ge 0$, regard $\mbox{Sym}^kX$ as formal positive sums $\sum \alpha_i x_i$ of distinct points of $X$ associated with some partition  $\alpha=(\alpha_1, \dots, \alpha_\ell)$ with $|\alpha|=k$.  Let 
 $\Delta_\alpha$ be the set of all such sums associated with  a given $\alpha$.  Then  $\Delta_\alpha$ is a manifold of dimension $4\ell(\alpha)$, and $\mbox{Sym}^kX$ is the  disjoint union of the sets $\Delta_\alpha$ over all $\alpha$ with  $|\alpha|=k$.

With these preliminaries understood, we re-stratify the compactification \eqref{6.UC} by writing
\bear\label{7.newstat}
\ov{\M}_k\ =\ \M_{k}\ \cup\  \bigcup_{\alpha\in{\cal P}_k}\ \SS_\alpha,
\eear
where ${\cal S}_0=\M_k$ and 
$$
\SS_\alpha\ =\ \M_{k-|\alpha|}\times \Delta_\alpha.
$$
By the choice of ${\cal R}^o$,  the restriction $\SS_\alpha^o$ of   $\SS_\alpha$ over ${\cal R}^o$  is, for each $\alpha$,  a Banach manifold with a Fredholm projection $\pi_\alpha:\SS_\alpha^o \to {\cal R}^o$  of  index 
\be\label{A.iota}
\iota_\alpha \ =\ 2d(k-|\alpha|)+4\ell(\alpha)\ =\ 2d_k-4\Lambda(\alpha),
\ee
where $d_k$ is the index \eqref{6.d_k}.  

One then sees that conditions (a) and (b) of Lemma~\ref{LemmaA1} hold for the restriction of \eqref{7.newstat} over ${\cal R}^o$. To verify (c), suppose that a sequence $(A_n, \sum \alpha_i (x_n)_i)$ converges in the weak topology.  Then $\{A_n\}$ converges to a formal instanton $(B, \sum \beta_jy_j)$ with $B\in \M^o_{k-|\alpha|-|\beta|}$, and $\sum \alpha_i (x_n)_i$ converges to $\sum \gamma_m z_m$ with $\ell(\gamma)\le\ell(\alpha)$ and $|\gamma|=|\alpha|$.  Thus the limit is
$$
\left(B, \sum \beta_jy_j +\sum \gamma_m z_m\right) \in \M^o_{k-|\delta|} \times \Delta_\delta,
$$
with $\ell(\delta)\le \ell(\beta)+\ell(\gamma)\le  \ell(\alpha)+\ell(\beta)$ and $|\delta|=|\beta|+|\gamma|=|\alpha|+|\beta|$.   The level \eqref{A.defLambda} of this limit stratum is therefore
$$
\Lambda(\delta) \ =\ 2|\delta|-\ell(\delta)\ \ge\ \Lambda(\alpha)+\Lambda(\beta) \ge \Lambda(\alpha),
 $$
with equality if and only if $\beta=(0)$ and $\gamma=\alpha$. This, together with \eqref{A.iota},  implies property (c) of Lemma~\ref{LemmaA1}.  The proposition follows. 
\end{proof}

\vspace{5mm}
\setcounter{equation}{0}
\section{Relative fundamental classes and Donaldson polynomials}
\label{section8}
\bigskip

As  in Section~7, the universal moduli space \eqref{ASDmodulidiagram} of anti-self-dual instantons on a 4-manifold $X$ admits a compactification, the Uhlenbeck compactification  $\ov\pi:\ov\M_k\to{\cal R}$.  Under the assumptions of Lemmas~\ref{lemma7.1} and \ref{donlemma},    there is an open dense subset ${\cal R}^o$ of ${\cal R}$ and  a diagram 
 $$\begin{gathered}
\label{familymapdiagram}
\xymatrix{
\ov\M_k^o\ar[d]^{\ov\pi^o}  \ \ar[r]\  &\ov\M_k\ar[d]^{\ov\pi}\\
{\cal R}^o \  \ar@{^(->}[r] \  &{\cal R}\
}
\end{gathered} $$
where   $\ov\M_k^o$ is the restriction of $\ov\M_k$ over ${\cal R}^o$,  and
\begin{itemize}
\item[(i)] $\ov\pi^o:\ov\M_k^o\to {\cal R}^o$ is a Fredholm-stratified thin family.  \\[-2mm]
\item[(ii)]   Every path in ${\cal R}$  is the limit of paths in ${\cal R}^o$.   
\end{itemize}
Let ${\cal R}^{reg}$ be the set of regular values  of  the family (i).   By the Sard-Smale theorem,  
 ${\cal R}^{reg}$ is dense in ${\cal R}^o$, and hence is dense in ${\cal R}$.

With this setup,  Lemma~\ref{Lemma4.3} and Theorem~\ref{3.existsVFC}  produce a relative fundamental class for   $\ov{\M}^o_k\to{\cal R}^o$.  In fact, Lemma~\ref{corelemma} gives a stronger conclusion:  it shows that the Uhlenbeck compactification is a relatively thin family over the entire space of metrics.  Thus we obtain  a  relative fundamental class for Donaldson theory:

 \medskip
 
 \begin{prop}
 \label{6.DonaldsonVFC}
Let  $X$ be a  closed, oriented 4-manifold with $b^+_2(X)>1$, and  let $E\to X$ a $U(2)$ vector bundle with instanton number $k$ and $c_1(E)$  odd.  Then
\smallskip
\begin{enumerate}
\item[(a)] The Uhlenbeck compactification is a relatively thin family  with index $2d_k$ with ${\cal R}^*$ equal to  ${\cal R}^{reg}$ and 
 $d_k$ given by \eqref{6.d_k}. \\[-3mm]

\item[(b)] A homology orientation for $X$ determines a  relative fundamental class 
\bear\label{8.rfc}
[\ov\M_k(g)]^{rel} \in\H_{2d_k}\left(\ov\M_k(g)\right),
\eear
where $\ov\M_k(g)$ is the fiber of $\ov\M_k$ over $g\in{\cal R}$.
\end{enumerate}
\end{prop}

\vspace{4mm}
 
 To obtain invariants,  one would like, as in \eqref{3.Invariants},  to consider pairings
 $$
 \left\langle  \alpha,\ [\ov\M_k(g)]^{rel} \right\rangle
 $$
 where $\alpha$ is the restriction to $\ov\M_k(g)$ of a    \Cech cohomology class defined on ${\cal B}_k$.  Unfortunately,  this is not as straightforward as one might hope, and one must work harder. 
 
  Following Donaldson, the natural cohomology classes to consider are those in the image of the $\mu$-map
  $$
 \mu: H_2(X; \Q) \to \cHH^2({\cal B}_k^{irred}; \Q)
 $$
 (cf. Chapter~5 of \cite{DK}).  For each choice of classes $A_1, \dots, A_{d_k}\in H_2(X;\Q)$, the product $\mu(A_1)\cup \cdots \cup \mu(A_{d_k})$ restricts to a class
$$
\mu=\mu(A_1, \dots, A_{d_k}) \in \cHH^{2d_k}(\M^{irred}_k; \Q)
$$
whose dependence on the $A_i$  is multilinear and symmetric.  For each $g\in{\cal R}$, this further restricts under the inclusion $\iota_g:\M^{irred}_k(g)\hookrightarrow\M^{irred}_k$ of the fiber over $g$ 
 to a class
\be\label{5.defalpha}
 \iota_g^*\mu \in \cHH^{2d_k}(\M^{irred}_k(g); \Q).
\ee
But  these are not  classes in the cohomology of 
 $\ov\M_k(g)$, so cannot be  directly paired with the relative fundamental class.  Thus we proceed more indirectly.

The key observation is that, for each {\em regular} metric $g$,  the  classes  \eqref{5.defalpha} extend over the compatification $\ov\M_k(g)$ in a way that is consistent along paths.
(Here ``regular'' means $g\in {\cal R}^{reg}$, which is equivalent to conditions~9.2.4 and implies 9.2.13 in \cite{DK}.)  
One can then apply Extension Lemma~\ref{extLemma}  to obtain a relative fundamental class in 0-dimensional \Cech homology, which  yields invariants.  The remainder of this section gives the details.

\medskip

 \begin{lemma}
 \label{lemma6.3} 
 For each $A_1, \dots, A_{d_k}\in H_2(X;\Z)$, 
 \begin{enumerate}
\item[(a)] For each  $g\in {\cal R}^{reg}$, the class  \eqref{5.defalpha}, which depends on  $A_1, \dots, A_{d_k}$,  extends uniquely to an element $\mu_g$ of $\cHH^{2d_k}(\ov\M_k(g); \Q)$.\\[-3mm]
\item[(b)] There is a unique association  
 $$
 g\mapsto \alpha_g \in  \cHH_0(\ov\M_k(g); \Q)
 $$
   such that  
 \begin{enumerate}
\item[(i)]   for each $g\in {\cal R}^{reg}$, $\alpha_g$ is the cap product with the  fundamental class \eqref{8.rfc}:
\bear\label{8.cap}
   \alpha_g\ =\   [\ov\M_k(g)]^{rel}  \cap \mu_g.  
\eear 
\item[(ii)] the  consistency condition \eqref{8.consistency} below holds for every path $\gamma$ in ${\cal R}$.
\end{enumerate}
\end{enumerate}
 \end{lemma}
 
 \begin{proof} 
(a)  Donaldson and Kronheimer showed  \cite[Subsection~9.2.3]{DK} that for  each regular   $g$,  $\iota_g^*\mu$ has a \Cech representative with compact support in $\M_k^{irred}(g)$, which is equal to $\M_k(g)$ by Lemma~\ref{lemma7.1}(i).  Because $\ov\M_k(g)$ is a  Fredholm-stratified thin compactification  of $\M_k(g)$, the long exact sequence in \Cech cohomology, used as in the proof of Lemma~\ref{Lemma2.2}, shows that $\iota_g^*\mu$ extends uniquely   to a \Cech  class in the compactification
$$
\mu_g\in \cHH^{2d_k}(\ov\M_k(g);\Q).
$$

Furthermore, for each  regular path $\gamma$ in ${\cal R}^o$ with endpoints $g, g'$,  the pullback $\ov\M_k(\gamma)$ over $\gamma$  of the compactified moduli space contains no reducible connections and is a thin compactified cobordism as defined in Section~2.4 above.   Again as in \cite{DK}, the class $\iota_\gamma^*\mu$  has a representative compactly supported in $\M_k(\gamma)$, so extends uniquely  to a class $\mu_\gamma$ on $\ov\M_k(\gamma)$.   The uniqueness of these extensions implies that 
 \bear
 \label{5.consistency}
\mu_g=\iota_g^*\mu^\gamma \ \mbox{in}\   \cHH^{2d_k}(\ov\M_k(g); \Q)   \qquad  \mbox{and}\qquad   \mu_{g'}=\iota_{g'}^*\mu^\gamma\ \mbox{in}\   \cHH^{2d_k}(\ov\M_k(g'); \Q).
\eear

\smallskip

(b)   For each regular $g$,  define $\alpha_g$  to be the cap product \eqref{8.cap}.    By the naturality of cap products, \eqref{5.consistency} implies a  consistency condition for $\alpha_g$ of the form \eqref{3.consistent}, namely
\bear\label{8.consistency}
(\iota_0)_*\al_{g}\ =\ (\iota_1)_* \al_{g'} \quad \mbox{in} \quad \cHH_0(\ov\M_k(\gamma);\Q)
\eear
for every regular path $\gamma$.   Lemma~\ref{lemma7.1}(ii), together with the middle paragraph of the proof of  Lemma~\ref{corelemma},  shows that each path $\gamma$ in ${\cal R}$ with endpoints $g, g'\in {\cal R}^{reg}$ is a limit of paths $\gamma_k$  in ${\cal R}^o$ with the same endpoints. But each $\gamma_k$ is a limit of regular paths in  ${\cal R}^o$ with the same endpoints (cf. the proof of Lemma~\ref{Lemma4.3}), which means that the regular paths are dense in the space of all paths in ${\cal R}$ from $g$ to $g'$.  The hypotheses of Lemma~\ref{extensionlemma} then hold for $g \mapsto \alpha_g$, with $\P^*$ taken to be ${\cal R}^{reg}$, and the conclusion of  Lemma~\ref{extensionlemma} gives (b).  
 \end{proof}

 \begin{rem}
Alternatively, one could work with the index~0 universal ``cutdown'' moduli spaces defined by \cite[(9.2.8)]{DK}, and regard the class $\alpha_g$ in \eqref{8.cap} as the relative fundamental class of the cutdown moduli space.
 \end{rem}

 \bigskip

We can now use the class $\alpha_g$ of Lemma~\ref{lemma6.3}, which depends on $A_1, \dots, A_{d_k}$, to define numerical invariants.  For each $g\in{\cal R}$ there is a map
$$
q_k(g): {\rm Sym}^{d_k}H_2(X;\Q) \to \Q
$$
 defined by evaluating $\alpha_g$ on the class $1\in \cHH^0(\ov\M_k(g);\Q)$:
\bear\label{qkofg}
 q_k(g)\ =\ \langle 1, \alpha_g\rangle.
\eear
 
 \begin{prop}
 \label{Prop8.3}
 The map $q_k(g)$ is independent of $g\in{\cal R}$, and is equal to Donaldson's polynomial invariants.
 \end{prop}
 
 \begin{proof}
 First note  that the space ${\cal R}$ of Riemannian metrics is path-connected;  in fact, it is contractible. The consistency condition \eqref{8.consistency}  then shows that $q_k(g)$ is independent of $g$, exactly as in the proof of Corollary~\ref{Cor4.last}.  For regular $g$, we can use \eqref{8.cap} to rewrite  \eqref{qkofg}  as
 $$
 q_k(A_1, \dots, A_{d_k})(g)\ =\ \left\langle \mu_g,\ [\ov\M_k(g)]^{rel} \right\rangle  \ =\ \left\langle \iota_g^*\mu(A_1, \dots, A_{d_k}),\ [\M^{irred}_k(g)]\right\rangle,
 $$
where the last term is a pairing between a compactly supported cohomology class and the fundamental class of a non-compact manifold.
 This agrees  with Donaldson's  definition of $q_k$:  see Section~9.2 of \cite{DK}, especially (9.2.18)  and  the top of page 360.  
 \end{proof}

 Proposition~\ref{Prop8.3}  is a re-casting of Donaldson's theorem \cite{D1} in the form presented in \cite{D2}: it  implies that  the Donaldson polynomials are  invariants of the smooth structure of the manifold $X$,  depending on the class $c_1(E)$,  the orientation, and the homology orientation.  In fact, changes in $c_1(E)$ and the homology orientation  change  the Donaldson polynomial in a specific way \cite{MM}.     In the literature, the  story is completed by removing the assumption that $c_1$ is  odd by using the stabilizing trick of Morgan and Mrowka; see \cite{MM} or \cite[Section 6.3]{D2}.

 \medskip
 
 This viewpoint makes clear that the invariance of the Donaldson polynomials follows directly from two core facts:  (i) the Uhlenbeck compactification is a Fredholm-stratified thin family over an open, dense,  path-connected  subset ${\cal R}^o$ of the space of metrics, and (ii) $2d_k$-dimensional products of classes $\mu(A_i)$ extend to the compactification  of regular fibers.  Both appear  explicitly in the work of Donaldson.  As we have seen, these same two facts imply the existence of a relative fundamental class {$[\ov{\M}_k(g)]^{rel}$   defined for every metric $g$  and every $k$.

\vspace{.5cm}
\bigskip
\setcounter{equation}{0}
\section{Gromov--Witten theory}
\label{section9}
\bigskip

In the remaining two sections, we consider  thin compactifications in  Gromov--Witten theory.  This section summarizes the well-known setup; details can be found  in  \cite{ms2},  \cite{rt1},  \cite{rt2}, and \cite{IP}.  Throughout, we work in the stable range $2g-2+n> 0$.

The Deligne-Mumford spaces $\ov\M_{g,n}$ are  at  the foundation of Gromov--Witten theory.  Points in $\ov\M_{g,n}$ represent equivalence classes $[C]$ of stable,  connected nodal complex curves $C$ of arithmetic genus $g$ with $n$ marked points $x_1,\dots x_n$; those without nodes form the principal stratum $\M_{g,n}$.  There is a universal curve
\bear
\label{7.UCdiagram}
\xymatrix{
\ov{\U}_{g,n} \ar[d]^{\pi}\\
\ov{\M}_{g,n}
}
\eear
with the property that for each  stable curve $C$ as above there is a map $C\to \ov{\U}_{g,n}$ whose image is a fiber of \eqref{7.UCdiagram} that is biholomorphic (as a marked curve) to $C/\Aut(C)$.  More generally, for any connected, $n$-marked genus $g$ nodal curve $C$,  there is a     map 
\bear\label{8.psi}
\phi: C\to \ov{\U}_{g,n}
\eear
defined as the composition $C\to C^{st}\to \ov\U_{g,n}$ where $C^{st}$ is the stable curve (the stable model of $C$)  obtained by  collapsing  all unstable irreducible components of $C$, and  the second map is as above.

Now fix   a closed symplectic manifold $(X,\w)$,  a large   integer  $l$ and a number $r>2$.
As in Section 3.1 of \cite{ms2}, let   $\J$ be the smooth separable Banach manifold of all $C^l$   $\w$-tame almost complex structures $J$ on $X$.  We consider maps $f:C\to X$ whose domain is an $n$-marked   connected nodal curve $C$ with complex structure $j$.
Such a map is called $J$-holomorphic if 
$$
\del_Jf=\tfrac12(df+J df j)=0,
$$
and two such maps are regarded as equivalent if they differ by reparametrization.  
Let $\M_{A,g,n}(X)$ denote the moduli space of all equivalence classes    $([f], J)$ of  pairs $(f, J)$, where 
 $J\in \J$  and $f$ is a  $J$-holomorphic map of Sobolev class $W^{l,r}$  with smooth stable domain that represents $A=[f(C)]\in H_2(X; \Z)$. 
One then has a continuous  projection $\pi$ and a continuous  stabilization-evaluation map $se$
\bear
\label{eq6.3}
\xymatrix{
\M_{A, g,n}(X) \ar[d]^{\pi} \ar[r]^{se\ \ } &\ov{\M}_{g,n}\times X^n \\
\J  &
}
\eear
defined by $\pi(f, J)=J$ and $se(f,J)= ([C], f(x_1),\dots, f(x_n))$.

More generally, each map $f:C\to X$ from a connected nodal curve has an associated graph map
\bear
\label{6.graphmap}
F=F_f:  C\to  \ov{\cal U}_{g,n} \times X  
\eear
defined by $F(x)=(\phi(x),\, f(x))$; this is an embedding if $\Aut (C)=1$. Following Ruan and Tian  \cite{rt2}, we  can use $F$ to expand the base of \eqref{eq6.3}, as follows.

The universal curve $\ov\U_{g,n}$   is projective;   fix  a holomorphic  embedding $\ov\U_{g,n}\hookrightarrow {\Bbb P}^M $.    
 For each fixed  almost complex structure $J$, consider sections $\nu$ of the bundle $\Hom(\pi_1^* T {\Bbb P}^M,  \pi_2^* TX)$ over ${\Bbb P}^M\times X$ that satisfy $J\circ \nu+\nu\circ j=0$,  where  this $j$ is the complex structure on $\PP^M$.   The space of all $C^l$ pairs $(J, \nu)$ of this form is  also a  smooth separable Banach manifold, which  we denote  by $\JV$.
Each  $(J,\nu)\in\JV$ defines a deformation $J_\nu$ of the product almost complex structure on  ${\Bbb P}^M\times X$, and therefore on $ \ov\U_{g,n} \times X$, by writing 
\bear\label{3.J.nu}
J_\nu= \l(\begin{matrix} j&0\\ -\nu\circ j &J\end{matrix}\r).
\eear 
We identify such $J_\nu$ with the pair $(J,\nu)$ and call it a {\em Ruan-Tian perturbation}.

A  map $f:C\to X$ is  $(J,\nu)$-holomorphic if its graph satisfies $\del_{J_\nu} F=0$,  or equivalently if $f$ satisfies
\bear\label{del=nu}
\del_J f(x)=\nu(\phi(x), f(x)).
\eear
Such a map  is called  {\em stable} if, for each irreducible component $C_i$  of $C$, either $C_i$  is stable or $f(C_i)$ is  not a single point.

 The map $J\mapsto (J,0)$ induces  a  smooth inclusion $\J\hookrightarrow \JV$ of Banach manifolds. Furthermore,   the maps $\pi$ and $se$ extend continuously over the universal moduli space $\ov{\M}_{A,g,n}(X)$ of all triples $(f, J,\nu)$ where $f$ is a stable $(J,\nu)$-holomorphic map, giving  continuous maps
\bear
\label{7.stableJV}
\xymatrix{
\ov{\M}_{A, g,n}(X) \ar[d]^{\ov\pi} \ar[r]^{se\ \ } &\ov{\M}_{g,n}\times X^n \\
\JV. &
}
\eear

  \medskip
  
  The analysis of these maps is standard;  see, for example,   Chapter~3 of \cite{ms2}, Section~3 of \cite{rt2}, and Sections~4 and 5.1 of \cite{IP}.  Let $\E^{m,r}$ (resp. $\F^{m,r}$)  denote the space of $W^{m,r}$ sections of the bundle $f^*TX$ (resp. $T^{0,1}C\otimes_\cx f^*TX$ ) over $C$.  The  space of first order deformations of the complex structure on $C$ is the finite-dimensional vector space  $H^{0,1}(TC)$. The linearization of the $(J, \nu)$-holomorphic map equation \eqref{del=nu}  at $(f,J,\nu)$ is a bounded linear   operator
$$
{\bf D}_{f,J_\nu}: \E^{m,r}\times  H^{0,1}(TC) \times T_{J_\nu}\JV \longrightarrow  \F^{m-1,r}
$$
given by formula  \cite[(3.10)]{rt2}; see also \cite[Prop~3.1.1]{ms2}.  The elliptic theory of this operator leads to  two important regularity properties:

\smallskip
\begin{enumerate}
\item[{\bf Reg 1}.] If ${\bf D}_{f,J_\nu}$ is surjective  and $\Aut (f)=1$, the universal  moduli space $\pi:\M_{A,g,n}(X)\to\JV$ in \eqref{eq6.3} is   a manifold near $(f,J,\nu)$ with a natural relative orientation  (see the proofs of \cite[Theorem~3.2]{rt2} or \cite[Theorem~3.1.5]{ms2}).  \\[-2mm]
\item[{\bf Reg 2}.]   If {\bf Reg 1} holds, then at each regular value $(J,\nu)$ of $\pi$, the fiber $\M_{A,g,n}^{J,\nu}(X)$ is a manifold whose dimension is the index of $D_{f,J_\nu}$, which is
\be
\label{7.index}
\iota(A,g,n)=2[c_1(A)+(N-3)(1-g)+n]
\ee
where $\dim X=2N$. 
\end{enumerate}
\smallskip

 \bigskip

The construction of  Gromov--Witten invariants now hinges on a single issue:  can one find a thin compactification of \eqref{eq6.3} so that the map $se$ extends  over the compactification to give  diagram   \eqref{7.stableJV}?
Doing so, even over a portion of $\JV$, allows us to define the Gromov--Witten numbers
\be
\label{6.GWnumbers}
GW_{A,g,n}(\alpha)\ =\ \left\langle (se)^*\alpha  ,\  [\ov{\M}^{J}_{A,g,n}]^{rel}\right\rangle \qquad \mbox{for all $\alpha\in \cHH^*(\ov{\M}_{g,n}\times X^n; \Q)$}.
\ee
Note that $\ov{\M}_{g,n}\times X^n$ is locally contractible, so by \eqref{1.CechSingular}  $\alpha$ can also be regarded as an element of rational singular cohomology.

\bigskip

More specifically, assuming the existence of a thin compactification, we can  apply the results of  Sections~1-6, with the following payoffs:

\smallskip

\begin{enumerate}[(a)]
\item  A thin compactification for the fiber over a single regular $J\in \J$ yields a relative fundamental class $[\ov{\M}^{J}_{A,g,n}]^{rel}$.  However, the numbers \eqref{6.GWnumbers} may not be invariant under changes in $J$.

\item  A thin compactification over a connected  neighborhood $\P$ of $J$ gives a relative fundamental class  at each $J\in\P$, and  by Corollary~\ref{Cor4.last} the numbers  \eqref{6.GWnumbers} are independent of $J$ in $\P$.

\item  A thin compactification over all of $\J$ or $\JV$ gives numbers   \eqref{6.GWnumbers} that   depend only on the symplectic structure of $(X,\w)$.

\item A thin compactification over the larger space $\J_{symp}$ of all tame pairs $(\w, J)$, completed in an appropriate Sobolev norm,  implies that  the  numbers   \eqref{6.GWnumbers} are invariants of the isotopy class of the symplectic structure on $X$.
\end{enumerate}

\medskip

We will give examples of this procedure in the next section.  Before proceeding, here are some  simple examples that illustrate the ideas of this section.
 
\medskip

\begin{example}[\sc Rational ghost maps]
\label{6.rationalghostex}
{\rm  For each $J\in \J$, every $J$-holomorphic map $f:S^2\to X$ representing the trivial class $A=0$  is a constant map.  It follows that $D_{f,J}$ is the $\del$ operator on the  trivial holomorphic bundle  $f^*TX$, and $f$ is regular  because   the sheaf cohomology group $H^1(S^2, f^* TX)$ vanishes.   Hence for $n\ge 3$ the fibers of the moduli space 
 $\ov \M^J_{0, 0, n}(X)\to\J$ are all regular and canonically identified with $\ov{\M}_{0,n}\times X$.   The relative fundamental class $[\ov\M^J(X)]^{rel}$ is therefore equal to the actual fundamental class $[\ov\M_{0, n}\times X]$ and the GW invariants \eqref{6.GWnumbers} are independent of $J\in\J$.}
 \end{example}
 
\smallskip

\begin{example}[\sc K3 surfaces]  {\rm Let $X$ be a K3 surface, and consider  the moduli space $\M(X)\to \J_{alg}$ of smooth rational holomorphic maps $(f,J)$ for algebraic $J\in\J$.
By a theorem of Mumford and Mori (see \cite{MMu}),  every algebraic K3 contains  a non-trivial rational curve, so  for each algebraic $J$ the fiber $\M^J_{A,0,0}(X)$ is non-empty for some $A\not= 0$. But by  \eqref{7.index} the index  $\iota(A,0,0)=-2$ is negative.    Thus $\M_{A,0,0}(X)\to \J_{alg}$ does not satisfy condition $\mbox{\bf Reg~1}$  for any algebraic $J$.

 Now  expand the base by considering $\pi:\M(X)\to \J_{cx}$ over the space of all integrable  almost complex structures.  Each $J\in \J_{cx}$ determines a 20-dimensional subspace $H^{1,1}(X;\R)$ of $H^2(X;\R)\cong \R^{22}$, and the resulting map $\J_{cx}\to \mbox{Gr}(20,22)$ to the Grassmannian is a submersion.  But $A\in H_2(X;\Z)$ can be represented by a $J$-holomorphic curve only if the Poincar\'{e} dual  of $A$ is an integral $(1,1)$ class.  It follows that $\M^J_{A,g,n}(X)$ is empty for all $J$ in a subset $\P\subset \J_{cx}$ whose complement  is a locally finite countable union of codimension~2  submanifolds.  Since empty fibers are regular, a relative fundamental class exists over $\P$ and is equal to 0.    Lemma~\ref{corelemma} then applies, showing that
 $$
[\ov\M^J_{A,g,n}(X)]^{rel} = 0
$$
for all  $A\not=0$, $g$ and $n$, and all $J\in \J_{cx}$, including the algebraic $J$. 
 }
\end{example}

\smallskip

\begin{example}[\sc Convex manifolds] {\rm  A complex algebraic  manifold $(X, \w,  J)$ is called {\em convex} if  $H^1(C, f^*TX)=0$ for stable $J$-holomorphic maps $f:S^2\to X$.  Examples include projective spaces, Grassmannians, and Flag manifolds.   Convexity implies that all $J$-holomorphic maps with smooth domain are regular, so $\M^J_{A,0,n}(X)$ is smooth and complex.   It is also a quasi-projective variety (cf. \cite{fp}), so its closure is a thin compactification.  Hence  by  Lemma~\ref{Lemma2.2}        there is a relative fundamental class  $[\ov{\M}^J_{A,0,n}(X)]^{rel}$ for the given   $J$; more work is needed to determine if the associated GW numbers \eqref{6.GWnumbers} are symplectic invariants.}
\end{example}

\vspace{5mm}

\setcounter{equation}{0}
\section{Moduli spaces of stable maps}  
\label{section10}
\bigskip

The space of stable maps is the most commonly-used compactification of the moduli space \eqref{eq6.3} of smooth pseudo-holomorphic maps.  Indeed, it is often regarded as the central object of Gromov--Witten theory. This section uses existing  results to show that,  in certain rather special circumstances, the space of stable maps is a thin compactification over  parts of $\J$  or $\JV$.  In these cases, the space of stable maps carries  a relative fundamental class.

\smallskip

Each stable map $f:C\to X$ has an associated dual graph $\tau(f)$, whose vertices correspond to the irreducible components $C_i$ of $C$ and whose edges correspond to the nodes of $C$. Each vertex of the graph is labeled by the homology class $A_i=[f(C_i)]\in H_2(X;\Z)$, by the genus $g_i$ of $C_i$, and by the number $n_i$ of marked points on $C_i$.  Every such graph $\tau$  defines a stratum ${\cal S}_\tau$ consisting of all stable maps $f$ with $\tau(f)=\tau$.   The trivial graph, which consists of a single vertex and no edges, corresponds to the moduli space  $\M_{A,g,n}$  in \eqref{eq6.3}.
The  universal moduli space of all stable maps is  then the disjoint union
\be\label{8.1}
\ov{\M}_{A,g,n}\ =\ \ \M_{A,g,n} \ \cup\ \bigcup {\cal S}_\tau,
\ee
where the  last union is over all non-trivial graphs $\tau$ with $\sum A_i=A$, $\sum n_i=n$, and with $\sum g_i$ plus the  first Betti number of the graph equal to $g$.  The Gromov Compactness Theorem (cf. \cite{is-GCT}) implies that the projection $\ov{\pi}: \ov{\M}_{A,g,n} \to \J$ is proper.

To check whether \eqref{8.1} is a  thin compactification one must, as always, compute the index of     the restriction $\pi_\tau:{\cal S}_\tau\to\JV$  of $\ov\pi$ to each stratum ${\cal S}_\tau$,  and prove transversality results that show that ${\cal S}_\tau$ is a manifold over $\J$.  In this case, the index calculations have been done many times in the literature (for example, see Theorem~6.2.6(i) in \cite{ms2} or Section~4 in \cite{rt1} for the $g=0$ case, and Section~3 in \cite{rt2} in general). These calculations show that, for each  $\tau$, 
\be\label{8.index}
\ind \ \pi_\tau\  = \ \iota(A,g,n) - 2k
\ee
where $\iota(A,g,n)$ is the index \eqref{7.index} of the principal stratum $\pi: \M_{A,g,n}\to\J$,  and $k$ is the number of nodes of the domain.  Lemma~\ref{LemmaA1} then shows that  \eqref{8.1} is a   Fredholm-stratified  thin compactification of the principal stratum {\em provided all  strata satisfy the transversality condition {\bf Reg~1} in Section~\ref{section9}}.

\smallskip

Unfortunately, transversality can only be shown for certain classes of stable maps.  In the remainder of this section, we  examine two  such classes of maps.

\subsection{\sc Moduli spaces of somewhere injective maps}

A stable map $f:C\to X$ is called  {\em somewhere injective} (si) if each irreducible component $C_i$ of $C$ contains a non-special point $p_i$ such that 
$$
(df)_{p_i}\not= 0 \qquad\mbox{and} \qquad f^{-1}(f(p_i))=\{p_i\}.
$$
(cf. \cite[Section 2.5]{ms2}). 
 In the literature, it is usual to consider the universal  moduli space of stable maps $\ov{\M}\to \J$,  and to show that the subset $\ov{\M}^*$ consisting of  somewhere injective maps has good properties.  We will  shift perspective:  {\em  instead  of restricting to a subset of $\ov\M$, we restrict to the subset of $\J$ consisting of those  ``nice'' $J$ for which the entire  fiber $\ov{\M}^{J}$  consists of somewhere injective  maps.}
 
 \smallskip

Thus we fix $(A,g,n)$ and define  the  (possibly empty) subset  $\J_{si}\ =\ \J_{si}(A,g,n)$ of $\J$   by
$$
\J_{si}\ =\ \Big\{ J \in\J\ \Big|\  \mbox{all $(f,J)\in\ov{\M}_{A,g,n}^{J}$ are si}\Big\}.
$$
We then consider the map
\bear
\label{8.2modulispace}
\xymatrix{
\ov{\M}'_{A,g,n}(X) \ar[d]^{\ov\pi'}\\
\J_{si}
}
\eear
obtained by restricting the space of stable maps \eqref{7.stableJV} over $\J_{si}$,   with the stratification
$$
\ov{\M}'_{A,g,n}\ =\ \ \M'_{A,g,n} \ \cup\ \bigcup {\cal S}'_\tau,
$$
obtained by restricting \eqref{8.1} over $\J_{si}$.  First note that:

 \begin{lemma}
 \label{donlemma1}
$ \J_{si}(A,g,n)$ is an open subset of $\J$, so is  a Banach manifold.
 \end{lemma}

The proof of  Lemma~\ref{donlemma1} is given at the end of this subsection. Assuming it,    one obtains a relative fundamental class,  in any one of the homology theories \eqref{1.defH},  for the space of stable maps over $\J_{si}$: 

 \begin{prop}
 \label{GW-si-lemma}
The family \eqref{8.2modulispace} is a Fredholm-stratified  thin compactification whose index $d=\iota(A,g,n)$ is given by \eqref{7.index}.  It therefore  admits a  unique relative fundamental class  which, in  particular, assigns an element
$$
[\ov{\M}^J_{A,g,n}(X)]^{rel} \in \H_d\left(\ov{\M}^J_{A,g,n}(X)\right)
$$
to each $J\in \J_{si}(A, g, n)$.
 \end{prop}

 By Corollary~\ref{Cor4.last}, the corresponding GW numbers  \eqref{6.GWnumbers} are constant on each path-component of  $\J_{si}(A, g, n)$. 
 \begin{proof} 
 Following the discussion in Section~\ref{section9}, it suffices to verify the assumptions of {\bf Reg~1}.   First, observe that somewhere injective maps have no non-trivial automorphisms.
  Next, standard arguments show that for each somewhere injective   $f$,  one can use the variation in the parameter 
$J\in \J$ to show that the linearization of the equation $\del_{J} f=0$ (with fixed domain and map $f$) is onto.  Specifically, for the $g=0$ case,  Proposition~6.2.7 and Theorem~6.3.1 in \cite{ms2}  imply that each stratum   ${\cal S}'_\tau$ of  \eqref{8.2modulispace} is a Banach manifold and  $\pi'_\tau:{\cal S}'_\tau\to \J_{si}$ has index given by \eqref{8.index}.   As mentioned  before \eqref{7.index}, the principal stratum is relatively oriented. Therefore \eqref{8.2modulispace} is a Fredholm-stratified thin compactification
 when $g=0$. 

The same proofs (Propositions~6.2.7 and 6.2.8 and the proof of Theorem~6.3.1) in \cite{ms2} also apply for $g>0$:  they show that the  linearization is surjective using variations that fix the complex structure on the domain, which implies, {\em a fortiori}, surjectivity   as the domain is allowed to vary.  
\end{proof}

 While Proposition~\ref{GW-si-lemma}  implies that the Gromov--Witten numbers are invariant under small deformations of $J$,  it does not imply that they are symplectic invariants  unless one can show that $\J_{si}$ is equal  to $\J$,  or at least  is path-connected, and open and  dense in $\J$.   The following examples give two simple cases where this occurs.

 \medskip

 \begin{ex} \label{example7.5}
 For $X={\Bbb CP}^N$, the universal space $\ov{\M}_{L, 0,0}(X)$ of stable rational maps representing the class of a line is smooth and equal to $\M_{L, 0,0}(X)$, and $\J_{si}(L,0,0)$ is all of $\J$.
 \end{ex}

\begin{ex}\label{example7.6}
Assume $X$ is a Calabi-Yau 3-fold.  As in \eqref{8.2modulispace}, consider the universal moduli space 
$\ov{\M}'_{A,0,0}(X) \to \J_{si}$ of unmarked stable rational curves representing a primitive homology class $A\in H_2(X;\Z)$.
 In this case, $\ov\pi$ has index~0 and, we claim, $\J_{si}=\J_{si}(A, 0,0)$ is  not only open, but is also dense and path-connected.  Hence Proposition~\ref{GW-si-lemma} gives a  relative fundamental class 
 $$
[\ov\M^J_{A,0,0}(X)]^{rel}\,\in\, \H_0\big(\ov\M^J_{A,0,0}(X) \big)
$$
defined for all $J\in \J_{si}$, and therefore for all $J\in \J$ by the Extension Lemma~\ref{extLemma}.  Evaluating on $1\in \cHH^0(\ov\M^J)$ then gives a well-defined numerical GW invariant.  

\smallskip

To prove the claim, note that, by Corollaries~1.4 and 6.6  of \cite{IP}, there is a path-connected dense subset $\J_{isol}^E$ of  $\J$ (with $E=\omega(A)$) such that, for  each $J\in \J_{isol}^E$, all somewhere injective $J$-holomorphic maps with energy at most $E$ are embeddings, and their images are  disjoint.  Fix $J\in \J_{isol}^E$. Then by Lemma~1.5(a) of \cite{IP} any $J$-holomorphic map 
$f\in \ov\M^J_{A, 0,0}(X)$ factors as a composition $f=g\circ \phi$ of a holomorphic map $\phi:C\ra C_{red}$ of (connected) complex curves and a $J$-holomorphic embedding $g:C_{red}\ra X$. But $A$ is primitive so the degree of $\phi$ is 1, and $C$ is an unmarked rational curve, so $\phi$ cannot have any constant components. Therefore $f$ is an embedding of a smooth curve; in particular, $f$ is somewhere injective.  Thus $ \J_{isol}^E\subseteq \J_{si}$.  But this means that $\J_{si}$  is an open subset of the manifold $\J$ that   contains a dense path-connected set.  It follows that $\J_{si}$ itself  is dense and path-connected, as claimed.   
\end{ex}

\medskip

 We conclude this subsection by supplying the deferred proof.

\begin{proof}[Proof of Lemma~\ref{donlemma1}] 
 From the discussion in \cite[Section 2.5]{ms2}, one  sees that the complement of $\J_{si}$ in $\J$ is the set of all $J$ such that there exists a $J$-holomorphic map $f:C\to X$  in $\ov\M_{A, g, n}(X)$ and an irreducible component $C_i$ of $C$ with either
\hspace*{1cm} \begin{enumerate}[(i)] \itemsep=2mm
\setlength\itemindent{1cm} 
\item $f(C_i)=p$ is a single point,
\item the restriction $f|_{C_i}$ is a multiple cover of its image, or
\item there is another component $C_j$ of $C$ with $f(C_i)=f(C_j)$. \\[-4mm]
\end{enumerate}
We will show that each of these is a closed condition on $J$, so the complement of $\J_{si}$ is the union of three closed sets.

Suppose that  a sequence $\{J_k\}$  converges to $J\in \J$ and that there are  stable $J_k$-holomorphic maps $f_k:C_k\to X$ and components $C'_k\subset C_k$  with $f_k(C_k')=p_k$ as in (i). By Gromov compactness,  after passing to a subsequence and then a diagonal subsequence,   $\{f_k\}$ and  $\{f_k|_{C'_k}\}$    converge to  $J_0$-holomorphic maps $f:C\to X$ and  $f':C'\to X$, respectively, for some nodal curve $C$ and subcurve $C'$ with $f'=f|_{C'}$.  But then  $f'$  is a constant map.  Thus (i) is a closed condition on $J$. 

If each $\{f_k|_{C'_k}\}$ is multiply covered then, by the proof of \cite[Proposition~2.5.1]{ms2},  there exist curves $B_k$ and   holomorphic maps $\phi_k: C_k'\ra B_k$ of degree $>1$ such that $f_k|_{C_k'}$ is the composition $g_k\circ \phi_k$ for some $J_k$-holomorphic map $g_k:B_k\to X$.  Again by Gromov compactness, we may assume that, after restricting to $C'_k$, these converge to maps $f'$, $g$ and $\phi$ with $f'=g\circ\phi$ and $\deg \phi>1$.
Then  $f'=f|_{C'}$ satisfies (ii), so  (ii) is a closed condition on $J$.

The proof for (iii) is similar after using \cite[Corollary 2.5.3]{ms2} to write $f_k|_{C_i}$ as the composition of $\phi_k:C_k^i\to C_k^j$ and $g_k:C_k^j\to X$.
\end{proof} 
\bigskip

\subsection{\sc Moduli spaces of domain-fine maps} The somewhere injective condition is  too restrictive for most applications.  In the genus~0 case, the needed
transversality results hold for the slightly larger class of  maps (``simple maps'') that are somewhere injective on the complement of   ghost  components;  see \cite[Example~6.2.5]{ms2}.   But it is more effective to expand the base space $\J$ to the space $\JV$ of Ruan-Tian perturbations  and work with the universal moduli space \eqref{7.stableJV} of $(J,\nu)$-holomorphic maps.  
 Here one has results analogous to those of the previous section for a different class of maps:

\begin{defn}
\label{Def9.1}
A  $(J,\nu)$-holomorphic map $f:C\to X$ is called {\em domain-fine} if $\Aut \; C=1$. 
\end{defn}
Note that any domain-fine map $f:C\to X$ is a stable map.  Furthermore,  the map $C\mapsto \phi(C)$   defined by \eqref{8.psi}  is  an embedding, so the graph map \eqref{6.graphmap} is  also an embedding, and hence is somewhere injective.  While the proofs in the previous  subsection do not automatically apply  (because the set of  almost complex structures on $\ov\U_{g,n}\times X$ is restricted to be 
 of the form \eqref{3.J.nu}), their conclusions hold,  as we show next.

\smallskip

Again, we fix  $(A,g,n)$,  set
$$
\JV_{df}\ =\ \JV_{df}(A,g,n)\ =\ \Big\{ J\in\JV\ \Big|\  \mbox{all $(f,J)\in\ov{\M}_{A,g,n}^{J}$ are domain-fine}\Big\},
$$
and consider the map
\bear
\label{9.2.diagram}
\xymatrix{
\ov{\M}''_{A,g,n}(X) \ar[d]^{\ov\pi''}\\
\JV_{df}
}
\eear
obtained by restricting the space of stable maps \eqref{7.stableJV} over $\J_{df}$.  Then   \eqref{8.1} restricts  to a stratification
$$
\ov{\M}''_{A,g,n}\ =\ \ \M''_{A,g,n} \ \cup\ \bigcup {\cal S}''_\tau.
$$
Corresponding to Lemma~\ref{donlemma1}, we have:

\begin{lemma}
\label{JstableOpenLemma}
 $\JV_{df}$ is an open subset of $\JV$, so is a Banach manifold.

   \end{lemma}
\begin{proof}
 Under Gromov convergence,  the order of the automorphism group  of the domain is upper semi-continuous, and 
 limits of unstable domain components are unstable.  Thus each domain-fine map $f$ has a neighborhood with the same property. For $(J,\nu)\in \JV_{df}$, these open sets cover  the moduli space  $\ov\M^{J,\nu}_{A, g, n}(X)$,  and hence by compactness cover the moduli spaces $\pi^{-1}(U)$ for some open neighborhood $U$ of $(J,\nu)$.
\end{proof}

\medskip

Lemma~\ref{JstableOpenLemma} enables us to rephrase a result of Ruan and Tian in \cite{rt2} to show that the moduli space \eqref{9.2.diagram}  over $\JV_{df}$ admits a relative fundamental class  in  the homology theories \eqref{1.defH}.

\begin{prop}  
\label{Ruan-Tian}
Fix $(A, g, n)$ and $\JV_{df}$ as above.   Then the restriction \eqref{9.2.diagram}  of the universal moduli space of stable maps   over $\JV_{df}$ is a  Fredholm-stratified  thin compactification  $\ov{\M}_{A,g,n}(X)\to \JV_{df}$ of index $d=\iota(A, g,n)$. Therefore it admits a unique relative fundamental class
$$
[ \ov\M_{A, g, n}^{J}(X)]^{rel} \in \H_*\left( \ov\M_{A,g,n}^{J}(X)\right).
$$

\end{prop}  

 Again, by Corollary~\ref{Cor4.last}, the corresponding GW numbers  \eqref{6.GWnumbers} are  invariant under small deformations of $(J, \nu)$, and  are constant on each path-component of  $\JV_{df}(A, g, n)$.

\begin{proof} For domain-fine maps $f$, we have $\Aut(f)=1$ and the graph map $F$ is an embedding.  Hence one can use the variation in $\nu$ to show that the linearization of the equation $\del _J f =\nu$ is onto,  as in  the proof of \cite[Proposition~3.2]{rt2}.  The proof is completed  exactly as the proof  of Proposition~\ref{GW-si-lemma}.  
\end{proof}

\begin{ex}
\label{Eg.ghost} 
Let $\ov \M_{0, 0, n}(X) \to \JV$ be the moduli space of stable $(J,\nu)$-holomorphic   rational  maps representing the class $0\in H_2(X)$ and with $n\ge 3$ marked points.
Because stable rational curves have no non-trivial automorphisms, maps of this type are domain-fine for all $(J,\nu)$.  Thus in this case $\J_{df}(0,0,n)$  is all of $\JV$.
\end{ex}

\medskip

In general, $\JV_{df}$ may not be all of $\JV$, and  may even be empty.    Proposition~\ref{Ruan-Tian} is then insufficient  to define symplectic invariants. This is a manifestation of a well-known problem in symplectic Gromov--Witten theory originally identified by Ruan and Tian:  the space of stable maps  may not be a relatively  thin family    because of the presence of multiply-covered unstable domain components. On such components, the perturbation $\nu$ vanishes and cannot be used to verify condition {\bf Reg~1} of Section~\ref{section9}.   In subsequent papers, we will extend and apply the constructions in Sections~1-6 with the aim of moving past this obstacle.

 \vspace{5mm}

\vspace{6mm}

\end{document}